\documentclass[a4paper,11pt,twoside]{amsart}

\usepackage{amsmath}
\usepackage{amsthm}
\usepackage{amsfonts}
\usepackage{amssymb}
\usepackage{mathrsfs}
\usepackage{graphicx}
\usepackage{enumerate}
\usepackage{url}

\usepackage{wasysym}

\usepackage{tikz}
\usepackage{tikz-cd}

\usepackage{subcaption}

\usetikzlibrary{arrows}
\usetikzlibrary{decorations.markings}

\tikzstyle{CCC}=[shape=circle, draw, fill=black!10, align=center, font=\scriptsize]
\tikzstyle{CC}=[shape=circle, draw, align=center, font=\scriptsize]

\theoremstyle{plain}
\newtheorem{lemma}{Lemma}[section]
\newtheorem{prop}[lemma]{Proposition}
\newtheorem{theo}[lemma]{Theorem}
\newtheorem{coro}[lemma]{Corollary}
\theoremstyle{remark}
\newtheorem{rem}[lemma]{Remark}

\theoremstyle{definition}
\newtheorem{definition}[lemma]{Definition}
\newtheorem{ex}[lemma]{Example}

\newcommand{\N}{\mathbb{N}}

\newcommand{\C}{\mathscr{C}}
\newcommand{\F}{\mathscr{F}}
\newcommand{\U}{\mathscr{U}}

\newcommand{\op}{\textup{op}}

\newcommand{\mxl[1]}{\textup{mxl}(#1)}
\newcommand{\mnl[1]}{\textup{mnl}(#1)}

\newcommand{\ub}[1]{#1^\sharp}
\newcommand{\lb}[1]{#1^\flat}
\newcommand{\CU}{{\bf C}_\U}
\newcommand{\CF}{{\bf C}_\F}

\newcommand{\id}{\textup{id}}
\newcommand{\inc}{\textup{inc}}
\newcommand{\hocolim}{\int}

\newcommand{\modu}[1]{\ \ \textnormal{mod}\ #1}

\begin{document}

\title{A new tool to study the fixed point property of finite posets}

\author{Ana Gargantini}
\address{Facultad de Ciencias Exactas y Naturales \\ Universidad Nacional de Cuyo \\ Mendoza, Argentina.}
\email{anagargantini@gmail.com}

\author{Miguel Ottina}
\address{Facultad de Ciencias Exactas y Naturales \\ Universidad Nacional de Cuyo \\ Mendoza, Argentina.}
\email{miguelottina@gmail.com}




\thanks{Research partially supported by grant M044 (2016--2018) of SeCTyP, UNCuyo.}

\begin{abstract}
We develop a novel tool to study the fixed point property of finite posets using a topological approach. Our tool is a construction which turns out to induce an endofunctor of the homotopy category of finite $T_0$--spaces. We study many properties of this construction and give several examples of application.
\end{abstract}

\maketitle

\section{Introduction}

In this article we develop a novel tool to study the fixed point property of finite $T_0$--spaces (or equivalently finite posets) which consists in analysing a new finite $T_0$--space which is constructed from the one whose fixed point property we want to study. More specifically, given a finite $T_0$--space $X$ we construct a new finite $T_0$--space $\C(X)$ whose elements are connected subspaces of $X$ and such that each continuous map $f\colon X\to X$ induces a continuous map $\C(f)\colon \C(X)\to \C(X)$. This assignment is not functorial but defines an endofunctor of the homotopy category of finite $T_0$--spaces.

We show with several examples that this tool is useful to study the fixed point property of finite $T_0$--spaces and we develop many properties of this construction so as to gain a better understanding of it. Among them, we prove that our construction preserves homotopy type and we study how it changes if we remove beat points of the given space.

\section{Preliminaries}

\subsection*{Finite topological spaces}

Let $X$ be a finite topological space. For each $x\in X$ the \emph{minimal open set} that contains $x$ is defined as the intersection of all the open subsets of $X$ that contain $x$ and is denoted by $U_x$ (or by $U_x^X$ if we need to make explicit the topological space in which it is considered). We denote by $F_x$ the intersection of all closed sets that contain $x$. We also denote $\widehat U_x = U_x-\{x\}$ and $\widehat F_x = F_x-\{x\}$. The set $\{U_x\mid x\in X\}$ is a basis for the topology of $X$. We can define a preorder in $X$ by $x_1\leq x_2$ if and only if $U_{x_1}\subseteq U_{x_2}$. Moreover, this defines a bijective correspondence between topologies and preorders in a given finite set, under which $T_0$--topologies correspond to partial orders \cite{alexandroff1937diskrete}. This correspondence allows us to consider a finite $T_0$--space as a finite poset, and a finite poset as a finite $T_0$--space. In addition, a map between finite $T_0$--spaces is continuous if and only if it is order-preserving.

Let $X$ be a finite topological space with associated preorder $\leq$. We define $X^{\op}$ as the finite space with underlying set $X$ and associated preorder given by $x\leq^{\op}y$ if and only if $y\leq x$ for all $x,y\in X$.

Let $X$ and $Y$ be finite $T_0$--spaces and let $f,g\colon X\to Y$ be continuous maps. Then $f$ and $g$ are homotopic (denoted by $f \simeq g$) if and only if there exist $n\in \N$ and continuous maps $f_0, f_1, \ldots, f_n\colon X \to Y$ such that $f = f_0\leq f_1\geq \ldots \leq f_n = g$ \cite[Corollary 1.2.6]{barmak2011algebraic}.

As a consequence, if $f$ and $g$ are such that for each $x\in X$, $f(x)$ and $g(x)$ are comparable, then 
$f\simeq g$ (since both $f$ and $g$ are comparable with the continuous map $h=\max\{f,g\}$).

Let $X$ be a finite $T_0$--space and let $x\in X$. We say that $x$ is a \emph{down beat point} of $X$ if $\widehat U_x$ has a maximum element. We say that $x$ is an \emph{up beat point} of $X$ is $\widehat F_x$ has a minimum element. If $x$ is a down beat point or an up beat point, we say that $x$ is a \emph{beat point} of $X$. It is well known that if $X$ is a finite $T_0$--space and $x_0\in X$ is a beat point, then $X-\{x_0\}$ is a strong deformation retract of $X$ \cite[Proposition 1.3.4]{barmak2011algebraic}.

\subsection*{bp--retracts}

We recall the following definition from \cite{cianci2019combinatorial}.

\begin{definition}[{\cite[Definition 3.1]{cianci2019combinatorial}}]
Let $X$ be a finite $T_0$--space and let $A\subseteq X$. We say that $A$ is a \emph{dbp--retract} (resp.\ \emph{ubp--retract}) of $X$ if $A$ can be obtained from $X$ by successively removing down beat points (resp.\ up beat points), that is, if there exist $n\in\N_0$ and a sequence $X=X_0\supseteq X_1\supseteq\cdots\supseteq X_n=A$ of subspaces of $X$ such that, for all $i\in\{1,\ldots,n\}$, the space $X_i$ is obtained from $X_{i-1}$ by removing a single down beat point (resp.\ up beat point) of $X_{i-1}$.
\end{definition}

Clearly, if $X$ is a finite $T_0$--space and $A$ is a dbp--retract of $X$ then the minimal elements of $X$ belong to $A$ since they can not be down beat points of any subspace of $X$. Similarly, if $X$ is a finite $T_0$--space and $A$ is an ubp--retract of $X$ then the maximal elements of $X$ belong to $A$.

Observe also that if $X$ is a finite $T_0$--space and $A\subseteq X$ then $A$ is an ubp--retract of $X$ if and only if $A^\op$ is a dbp--retract of $X^\op$.

The following results of \cite{cianci2019combinatorial} give characterizations of dbp--retracts which will be needed in this work. Clearly, similar results hold for ubp--retracts.

\begin{theo}[{\cite[Theorem 3.5]{cianci2019combinatorial}}]
\label{theo_dbpr_equivalences}
Let $X$ be a finite $T_0$--space, let $A$ be a subspace of $X$ and let $i\colon A\to X$ be the inclusion map. Then, the following propositions are equivalent:
\begin{enumerate}[(1)]
\item $A$ is a dbp--retract of $X$.
\item There exists a continuous function $f\colon X\to X$ such that $f\leq\id_X$, $f^{2}=f$ and $f(X)=A$.
\item There exists a unique continuous function $f\colon X\to X$ such that $f\leq\id_X$, $f^{2}=f$ and $f(X)=A$.
\item There exists a retraction $r\colon X\to A$ of $i$ such that $ir\leq\id_X$.
\item There exists a unique retraction $r\colon X\to A$ of $i$ such that $ir\leq\id_X$.
\end{enumerate}
\end{theo}

\begin{theo}[{\cite[Theorem 3.7]{cianci2019combinatorial}}]
\label{theo_Ux_cap_A}
Let $X$ be a finite $T_0$--space and let $A$ be a subspace of $X$. Then, $A$ is a dbp--retract of $X$ if and only if $U^{X}_x\cap A$ has a maximum for every $x\in X$. Equivalently, $A$ is a dbp--retract of $X$ if and only if $U^{X}_x\cap A$ has a maximum for every $x\in X-A$.
\end{theo}

From the proof of theorem 3.7 of \cite{cianci2019combinatorial} it follows that if $A$ is a dbp-retract of a finite  $T_0$--space $X$ and $r\colon X\to A$ is the corresponding retraction, then $r(x)=\max(U^X_x\cap A)$ for all $x\in X$.

\subsection*{Fixed point property}

Let $X$ be a topological space.
We say that a continuous map $f\colon X\to X$ has a \emph{fixed point} if there exists $x\in X$ such that $f(x)=x$.
We say that $X$ has the \emph{fixed point property{}} if all continuous maps $f\colon X\to X$ have a fixed point. It is easy to see that if a topological space fails to be $T_0$ or connected, then it does not have the fixed point propery. Thus, we will only consider connected finite $T_0$--spaces in this article.

It is known that if $X$ is a finite $T_0$--space, $f\colon X \to X$ is a continuous map and there exists $x\in X$ such that $f(x)$ and  $x$ are comparable, then $f$ has a fixed point. From this fact we deduce the following result.

\begin{prop}
\label{Uhat}
 Let $X$ be a finite $T_0$--space  and let $f\colon X\to X$ be a continuous map. If there exists $x \in X$ such that $\widehat{U}_x$ (respectively $\widehat{F}_x$) has the fixed point property, and $f(\widehat{U}_x) \subseteq U_x$ (respectively $f(\widehat{F}_x) \subseteq F_x$), then $f$ has a fixed point.
\end{prop}

\begin{proof}
Suppose that $f(\widehat{U}_x) \subseteq U_x$. If there exists $y \in \widehat{U}_x$ such that $f(y) = x$, then $f$ has a fixed point. Otherwise, $f(\widehat{U}_x) \subseteq \widehat{U}_x$ and hence $f$ has a fixed point.
\end{proof}

We recall the following result.

\begin{prop}{\cite[Corollary 10.1.4]{barmak2011algebraic}}
\label{prop_XheY_ppf}
 Let $X$ and $Y$ be homotopy equivalent finite $T_0$--spaces. If $X$ has the fixed point property, then $Y$ also has the fixed point property.
\end{prop}

It follows from the previous proposition that if $X$ is a finite $T_0$--space with the fixed point property and $x$ is a beat point of $X$, then $X-\{x\}$ also has the fixed point property. As a further consequence, if $X$ is a contractible finite $T_0$--space, then it has the fixed point property. Therefore, if a finite $T_0$--space has a maximum or minimum element, then it is contractible and hence it has the fixed point property.

Let $X$ be a finite $T_0$--space and let $n\in\N$ such that $n\geq 2$. We say that $X$ is a \emph{$2n$--crown} if $X$ is homeomorphic to the finite $T_0$--space given by the following Hasse diagram
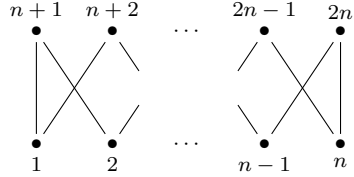
\begin{center}
\begin{tikzpicture}[y=1.5cm]
\tikzstyle{every node}=[font=\scriptsize]
\foreach \x in {1,2} \draw (\x,0) node(c\x)[inner sep=2pt]{$\bullet$} node[below=1]{$\x$}; 
\foreach \x in {3} \draw (\x,0) node(c\x)[inner sep=2pt]{$\cdots$};
\foreach \x in {4} \draw (\x,0) node(c\x)[inner sep=2pt]{$\bullet$} node[below=1]{${n-1}$}; 
\foreach \x in {5} \draw (\x,0) node(c\x)[inner sep=2pt]{$\bullet$} node[below=1]{$n$}; 

\foreach \x in {1,2} \draw (\x,1) node(b\x)[inner sep=2pt]{$\bullet$} node[above=1]{$n+\x$}; 
\foreach \x in {3} \draw (\x,1) node(b\x)[inner sep=2pt]{$\cdots$};
\foreach \x in {4} \draw (\x,1) node(b\x)[inner sep=2pt]{$\bullet$} node[above=1]{$2n-1$}; 
\foreach \x in {5} \draw (\x,1) node(b\x)[inner sep=2pt]{$\bullet$} node[above=1]{$2n$}; 

\foreach \x in {1,2} \draw (b1)--(c\x);
\foreach \x in {1} \draw (b2)--(c\x);
\draw[shorten >=1cm,shorten <=0cm] (b2)--(c3);
\draw[shorten >=1cm,shorten <=0cm] (c2)--(b3);
\draw[shorten >=1cm,shorten <=0cm] (b4)--(c3);
\draw[shorten >=1cm,shorten <=0cm] (c4)--(b3);
\foreach \x in {5} \draw (b4)--(c\x);
\foreach \x in {4,5} \draw (b5)--(c\x);
\end{tikzpicture}
\end{center}
For each element $a$ in a $2n$--crown $X$, we have that $X-\{a\}$ is contractible. From this fact we obtain the following proposition.

\begin{prop} \label{prop_crown}
Let $n\in\N$ such that $n\geq 2$ and let $X$ be a $2n$--crown. Let $f\colon X\to X$ be a continuous map. If $f$ is not bijective then $f$ has a fixed point.
\end{prop}

\section{The $\C$ construction }
\label{section-C_constr}

Let $S$ be a set. The power set of $X$ will be denoted by $\mathcal{P}(S)$ and the set $\mathcal{P}(S)-\{\varnothing \}$ will be denoted by $\mathcal{P}_{\neq\varnothing}(S)$.

\begin{definition}
Let $X$ be a finite $T_0$--space. Let $A$ be a non-empty subset of $X$. We define
\begin{displaymath}
U^X_A = \bigcap_{a\in A} U^X_a  \qquad \textnormal{and} \qquad F^X_A = \bigcap_{a\in A} F^X_a 
\end{displaymath}
\end{definition}

When there is no risk of confusion, the superindices $X$ of the previous definition will be omitted and thus $U^X_A$ and $F^X_A$ will be denoted by $U_A$ and $F_A$ respectively.

Note that if $X$ is a finite $T_0$--space and $A$ is a non-empty subset of $X$, then $U_A$ is an open subset of $X$ and $F_A$ is a closed subset of $X$.

From now on, if $X$ is a topological space, $\mathcal{C}(X)$ will denote the set of connected components of $X$. 

\begin{definition}
 Let $X$ be a finite $T_0$--space. We define
 \begin{align*}
  \U_0(X) &= \{U_A\subseteq X \mid A\in \mathcal{P}_{\neq\varnothing}(\mxl[X]) \wedge U_A\neq \varnothing \}\textup{ and}\\
  \F_0(X) &= \{F_A\subseteq X \mid A\in \mathcal{P}_{\neq\varnothing}(\mnl[X]) \wedge F_A\neq \varnothing \}.
 \end{align*}

We also define
\begin{align*}
 \U(X) &= \{C\subseteq X \mid C\in\mathcal{C}(U) \textup{ for some } U\in \U_0(X)\} \textup{ and}\\
 \F(X) &= \{C\subseteq X \mid C\in\mathcal{C}(F) \textup{ for some } F\in \F_0(X) \}.
 \end{align*}
\end{definition}

Note that $\U(X)$ and $\F(X)$ are finite sets.

\begin{rem}
\label{rem_cc_of_U_F_are_open_closed}
Let $X$ be a finite $T_0$--space.
 
(1) Clearly, the elements of $\F(X)$ are closed subsets of $X$. In addition, since finite topological spaces are locally connected it follows that the elements of $\U(X)$ are open subsets of $X$.
 
(2) For all $a\in \mxl[X]$, $U_a\in \U(X)$ since $U_a$ is connected. In a similar way, for all $b\in \mnl[X]$, $F_b\in \F(X)$.
\end{rem}

The following proposition states that the sets $\U(X)$ and $\F(X)$ of the previous definition are disjoint, except in a trivial case.

\begin{prop}
\label{prop_UF_disjuntos}
 Let $X$ be a finite $T_0$--space.
 \begin{enumerate}[(a)]
  \item If $X$ is connected and $\U(X)\cap \F(X)\neq\varnothing$, then $X$ has a maximum element and a minimum element.
  \item If $X$ has a maximum element then $\U(X) = \{X\}$.
  \item If $X$ has a minimum element then $\F(X) = \{X\}$.
  \item If $X$ is connected and $\U(X)\cap \F(X)\neq\varnothing$, then $\U(X) = \F(X) = \{X\}$.
 \end{enumerate}
\end{prop} 
 
\begin{proof}\ 

 (\emph{a}) Let $C\in\U(X)\cap \F(X)$. Then $C$ is non-empty, open and closed in $X$ by remark \ref{rem_cc_of_U_F_are_open_closed}. Since $X$ is connected, $C=X$.
 
 Since $C\in\U(X)$, there exists $U\in \U_0(X)$ such that $C$ is a connected component of $U$. Hence $U=X$. Since $X$ is finite and $X=U$ has at most one maximal element of $X$, then $X$ has a unique maximal element, that is, $X$ has a maximum element.
 
 In a similar way, it follows that $X$ has a minimum element.
 
 (\emph{b}) If $a=\max X$, then $\U_0(X) = \{U_a\}=\{X\}$ and hence $\U(X)=\{X\}$.
 
 (\emph{c}) Analogous to the proof of (\emph{b}).
 
 (\emph{d}) It follows immediately from (\emph{a}), (\emph{b}) and (\emph{c}).
\end{proof}

Observe that in the previous proposition, the hypothesis that $X$ is connected is necessary for parts (\emph{a}) and (\emph{d}). Indeed, consider the discrete space of two points $X=\{0,1\}$. Clearly $F_0=\{0\}=U_0$. Hence $\{0\}\in \U(X)\cap \F(X)$, but $X$ has neither a maximum element nor a minimum element. Therefore item ($a$) does not hold. In addition, it is easy to verify that $\U(X)=\F(X)=\{\{0\},\{1\}\}$, and thus item ($d$) does not hold.

\begin{definition}
\label{defi_orden_en_C(X)}
Let $X$ be a finite $T_0$--space. We define the relation $\leq$ in $\U(X)\cup \F(X)$ as follows.

Given $C_1,C_2\in \U(X)\cup \F(X)$, we will say that $C_1\leq C_2$ if and only if one of the following conditions holds:
\begin{itemize}
 \item $C_1, C_2\in \U(X)$ and $C_1\subseteq C_2$, 
 \item $C_1, C_2\in \F(X)$ and $C_1\supseteq C_2$,
 \item $C_1\in \F(X)$, $C_2\in \U(X)$ and $C_1\cap C_2 \neq \varnothing$.
\end{itemize}
\end{definition}

It is not difficult to prove that the previously defined relation $\leq$ is an order relation in $\U(X)\cup\F(X)$.

\begin{definition}
Let $X$ be a finite $T_0$--space.
\begin{itemize}
\item We define $\U(X)$ as the finite $T_0$--space associated to the poset $(\U(X),\subseteq)$.
\item We define $\F(X)$ as the finite $T_0$--space associated to the poset $(\F(X),\supseteq)$.
\item We define $\C(X)$ as the finite $T_0$--space associated to the poset $(\U(X)\cup\F(X),\leq)$, where $\leq$ is the order defined in \ref{defi_orden_en_C(X)}.
\end{itemize}
\end{definition}

Observe that $\U(X)$ and $\F(X)$ are subspaces of $\C(X)$.

The following is a simple example of the $\C$ construction.

\begin{ex} \label{ex_easy}
Let $X$ be the finite $T_0$--space given by the following Hasse diagram. 
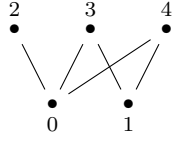
\begin{center}
\begin{tikzpicture}
\tikzstyle{every node}=[font=\scriptsize]
\foreach \x in {0,1} \draw (\x,0) node(\x){$\bullet$} node[below=1]{\x}; 
\foreach \x in {2,3,4} \draw (\x-2.5,1) node(\x){$\bullet$} node[above=1]{\x}; 
\foreach \x in {2,3,4} \draw (0)--(\x);
\foreach \x in {3,4} \draw (1)--(\x);
\end{tikzpicture}
\end{center}
Then the Hasse diagram of $\C(X)$ is
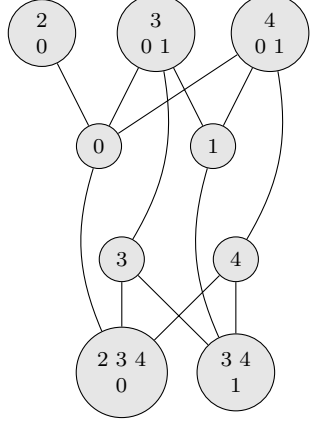
\begin{center}
\begin{tikzpicture}[x=1.5cm, y=1.5cm]
\path node (U2) at (0,3) [CCC] {2\\0}
node (U3) at (1,3) [CCC] {3\\0 1}
node (U4) at (2,3) [CCC] {4\\0 1}
 
node (U34a) at (0.5,2) [CCC] {0} 
node (U34b) at (1.5,2) [CCC] {1} 
 
node (F01a) at (0.7,1) [CCC] {3}
node (F01b) at (1.7,1) [CCC] {4}

node (F0) at (0.7,0) [CCC] {2 3 4\\0}
node (F1) at (1.7,0) [CCC] {3 4\\1};

\draw [-] (U34a) to (U3) to (U34b) to (U4) to (U34a) to (U2);
\draw [-] (F0) to (F01a) to (F1) to (F01b) to (F0);
\draw[-,bend left=20] (F0) to (U34a);
\draw[-,bend left=20] (F1) to (U34b);
\draw[-,bend right=20] (F01a) to (U3);
\draw[-,bend right=20] (F01b) to (U4);
\end{tikzpicture}
\end{center}
\end{ex}

The following proposition describes the maximal and minimal elements of the poset $\C(X)$ for any given finite $T_0$--space $X$.

\begin{prop}
 \label{prop_mnl_mxl_de_C(X)}
 Let $X$ be a finite $T_0$--space. Then \[\mxl[\C(X)] = \{U_a \mid a\in \mxl[X] \}\quad \textup{and}\quad \mnl[\C(X)] = \{F_a \mid a\in \mnl[X] \}.\]
\end{prop}

\begin{proof}
Let $a\in \mxl[X]$. Since $U_a$ is connected, $U_a\in \C(X)$. Let $C\in \U(X)$ such that $C\geq U_a$. There exists a non-empty subset $A\subseteq \mxl[X]$ such that $C$ is a connected component of $U_A$. Then $U_a\subseteq C\subseteq U_A$. It follows that $a\leq x$ for all $x\in A$, and thus $a = x$ for all $x\in A$. Hence $A=\{a\}$ and $U_a = U_A = C$. Therefore, $U_a\in \mxl[\C(X)]$.
 
Now suppose that $C\in \mxl[\C(X)]$. Suppose, in addition, that $C\in \F(X)$ and let $x\in C$. Since $X$ is finite, there exists $b\in \mxl[X]$ such that $x\leq b$. Hence $b\in C$ since $C$ is closed. Thus, $C\cap U_b\neq \varnothing$ and then $C<U_b$ which contradicts the maximality of $C$. 

Thus $C\in \U(X)$. Then there exists a non-empty subset $A\subseteq \mxl[X]$ such that $C$ is a connected component of $U_A$. Let $a\in A$. Then $C\subseteq U_a$, that is, $C\leq U_a$. Hence $C=U_a$.
 
The second identity can be proved in a similar way.
\end{proof}

The following proposition shows that the construction $\F$ can be interpreted as a dual version of $\U$.

\begin{prop}
\label{prop_C_op}
Let $X$ be a finite $T_0$--space.
\begin{enumerate}[(a)]
\item For all $A\subseteq \mnl[X]$, $F_A^X = \left(U_A^{X^{\op}}\right)^{\op}$. 
\item $\F(X) = \U(X^{\op})$ as sets.
\item $\C(X^{\op}) =\left(\C(X) \right)^{\op}$. In particular, $\F(X) = (\U(X^{\op}))^{\op}$ as posets with the corresponding partial orders induced from $\C(X)$ and $\C(X^{\op})$ respectively.
\end{enumerate}
\end{prop}

\begin{proof}\ 
 
(\emph{a}) Let $A\subseteq \mnl[X]$. Note that for all $a\in A$, $F_a = \left(U_a^{X^{\op}} \right)^{\op}$. Hence
\[
F_A^X =\bigcap_{a\in A} F_a = \bigcap_{a\in A}\left(U_a^{X^{\op}} \right)^{\op}= \left(\bigcap_{a\in A} U_a^{X^{\op}}\right)^{\op} = \left(U_A^{X^{\op}}\right)^{\op}.
\]
 
(\emph{b}) Follows from (\emph{a}).

($c$) It follows from ($b$) that $\U(X^\op)\cup \F(X^\op) = \F(X)\cup \U(X)$. Then the underlying sets of $\C(X^{\op})$ and $\left(\C(X) \right)^{\op}$ coincide. Let $\leq$ denote the partial order of $\C(X)$ and let $\preceq$ denote the partial order of $\C(X^{\op})$. Let $C_1,C_2\in \C(X^{\op})$. 
\begin{itemize}
\item If $C_1,C_2\in \U(X^{\op})=\F(X)$, then $C_1\preceq C_2$ if and only if  $C_1\subseteq C_2$ if and only if $C_2\leq C_1$ if and only if $C_1\leq^{\op} C_2$.
\item If $C_1,C_2\in \F(X^{\op})=\U(X)$, then $C_1\preceq C_2$ if and only if  $C_1\supseteq C_2$ if and only if $C_2\leq C_1$ if and only if $C_1\leq^{\op} C_2$.
\item If $C_1\in \F(X^{\op})$ and $C_2\in \U(X^{\op})$, then $C_1\preceq C_2$ if and only if  $C_1\cap C_2\neq \varnothing$ if and only if $C_2\leq C_1$ if and only if $C_1\leq^{\op} C_2$.
\end{itemize}
Hence, the partial orders $\preceq$ and $\leq^{\op}$ coincide.
\end{proof}

\begin{definition}
Let $X$ be a finite $T_0$--space and let $B\subseteq X$ be a subset. We define
\begin{displaymath}
\ub B=\{a\in\mxl[X]\mid B\subseteq U_a \}  \qquad  \textnormal{and} \qquad \lb B=\{a\in\mnl[X]\mid B\subseteq F_a \}.
\end{displaymath}
\end{definition}

\begin{rem}
Let $X$ be a finite $T_0$--space and let $B\subseteq X$ be a subset.

If $\ub B \neq \varnothing$ then $B\subseteq U_{\ub B}$. If, in addition, $B$ is a connected subspace of $X$ then there exists a unique connected component $C_0$ of $U_{\ub B}$ such that $B\subseteq C_0$.

Similarly, if $\lb B \neq \varnothing$ then $B\subseteq F_{\lb B}$, and if, in addition, $B$ is a connected subspace of $X$ then there exists a unique connected component $C_0$ of $F_{\lb B}$ such that $B\subseteq C_0$.
\end{rem}

\begin{lemma}
\label{lema-para_conexo_minimal}
 Let $X$ be a finite $T_0$--space and let $B\subseteq X$ be a connected subspace. Then:
 \begin{enumerate}[(a)]
  \item $\ub B \neq \varnothing$ if and only if $\{C\in\U(X)\mid B\subseteq C \} \neq\varnothing$.
  \item $\lb B \neq\varnothing$ if and only if $\{C\in\F(X)\mid B\subseteq C \} \neq\varnothing$.
 \end{enumerate}
\end{lemma}

\begin{proof}
 We will prove only item (\emph{a}). The proof of item (\emph{b}) is similar.
 
 ($\Rightarrow$) Let $a_0\in \ub B$. Then $U_{a_0}\in\{C\in\U(X)\mid B\subseteq C \}$.
 
 ($\Leftarrow$) Let $C_0\in\U(X)$ such that $B\subseteq C_0$ and let $A\subseteq \mxl[X]$ be a non-empty subset such that $C_0$ is a connected component of $U_A$. Let $a_0\in A$. Then $B\subseteq U_A\subseteq U_{a_0}$. Hence $a_0\in\ub B$.
\end{proof}

The following proposition states a key property of the $\C$ construction. 

\begin{prop}
\label{prop_conexo_minimal}
Let $X$ be a finite $T_0$--space and let $B\subseteq X$ be a connected non-empty subspace. 
\begin{enumerate}[(a)]
\item If $\{C\in \U(X) \mid B\subseteq C \}$ is non-empty, then it has a minimum element, which is the connected component of $U_{\ub B}$ that contains $B$.
\item If $\{C\in \F(X) \mid B\subseteq C \}$ is non-empty, then it has a maximum element, which is the connected component of $F_{\lb B}$ that contains $B$. 
\end{enumerate}
\end{prop}

\begin{proof}
By the previous lemma, $\ub B\neq\varnothing$. Let $C_0$ be the unique connected component of $U_{\ub B}$ such that $B\subseteq C_0$.

Let $C\in \U(X)$ such that $B\subseteq C$. There exists a non-empty subset $A\subseteq \mxl[X]$ such that $C$ is a connected component of $U_{A}$. Hence $B\subseteq U_a$ for all $a\in A$. Then, $A\subseteq \ub B$ and thus $U_{\ub B}\subseteq U_{A}$. Hence $C_0\subseteq U_{A}$. 
 
Since $C$ is a connected component of $U_{A}$ and $\varnothing\neq B\subseteq C_0 \cap C$, it follows that $C_0\subseteq C$, that is, $C_0\leq C$. Hence $C_0 = \min \{C\in \U(X) \mid B\subseteq C\}$.
 
The second part can be proved in a similar way. 
\end{proof}

The following corollary follows immediately from the previous proposition.

\begin{coro} \label{coro_cc_of_U_Bsharp_F_Bflat}
Let $X$ be a finite $T_0$--space.
\begin{enumerate}[(a)]
\item Let $B\in\U(X)$. Then $B$ is a connected component of $U_{\ub B}$.
\item Let $B\in\F(X)$. Then $B$ is a connected component of $F_{\lb B}$.
\end{enumerate}
\end{coro}

Proposition \ref{prop_conexo_minimal} allows us to give the following definition.

\begin{definition}
Let $X$ be a finite $T_0$--space and let $x\in X$. We define 
\begin{displaymath}
\CU(x)=\min \{C\in \U(X) \mid x\in C \}
\end{displaymath}
and 
\begin{displaymath}
\CF(x)=\max \{C\in \F(X) \mid x\in C \}.
\end{displaymath}
\end{definition}

\begin{prop}
Let $X$ be a finite $T_0$--space.
\begin{enumerate}[(a)]
\item The set $\U(X)$ is a basis for a topology in $X$ which is coarser than the topology of $X$.
\item Let $\mathcal{T}$ be the topology generated by $\U(X)$. The Kolmogorov quotient of $(X,\mathcal{T})$ is homeomorphic to a subspace of the finite $T_0$--space associated to the poset $\U(X)$.
\end{enumerate}
\end{prop}

\begin{proof}\ 

($a$) Since $U_a\in\U(X)$ for all $a\in\mxl[X]$ it follows that $\displaystyle\bigcup_{C\in\U(X)}C=X$. Now let $C_1,C_2\in\U(X)$ and let $x\in C_1\cap C_2$. Then $x\in\CU(x)\subseteq C_1 \cap C_2$ by \ref{prop_conexo_minimal}. Thus, $\U(X)$ is a basis for a topology in $X$ which, by \ref{rem_cc_of_U_F_are_open_closed}, is clearly coarser than the topology of $X$.

($b$) Let $K$ be the Kolmogorov quotient of $(X,\mathcal{T})$ and let $q\colon X\to K$ be the quotient map. Let $f\colon X\to \U(X)$ be defined by 
$f(x)=\CU(x)$. We will prove that $f$ is order-preserving. Let $a,b\in X$ be such that $a\leq b$ then $a\in\CU(b)$ since $\CU(b)$ is an open subset of $X$ which contains $b$. Thus $\CU(a)\subseteq \CU(b)$. Hence, $f$ is order-preserving, and therefore, it is a continuous map.

Now, let $c,d\in X$ be such that $q(c)=q(d)$. Thus, $\{C\in U(X) \mid q(c)\in C\}=\{C\in U(X) \mid q(d)\in C\}$ since $\U(X)$ is a basis for the topology $\mathcal{T}$. Hence, $f(c)=\CU(c)=\CU(d)=f(d)$. Therefore, there exists a continuous map $\overline{f}\colon K\to \U(X)$ such that $\overline{f}q=f$.

We will prove now that if $k,l\in K$ are such that $\overline{f}(k)\leq \overline{f}(l)$ then $k\leq l$. Let $k,l\in K$ be such that $\overline{f}(k)\leq \overline{f}(l)$. Let $x,y\in X$ be such that $k=q(x)$ and $l=q(y)$. Then $f(x)=\overline{f}(k)\leq \overline{f}(l)=f(y)$. Hence $\CU(x)\subseteq \CU(y)$. Then, for all $C\in \U(X)$ such that $y\in C$ we obtain that $x\in\CU(x)\subseteq \CU(y)\subseteq C$, and thus $x\in C$. Therefore $x\leq y$. Hence, $k=q(x)\leq q(y)=l$.
\end{proof}

\begin{definition} \label{def_Uf_Ff_Cf}
Let $X$ and $Y$ be finite $T_0$--spaces and let $f\colon X\to Y$ be a continuous map. We define 
\begin{displaymath}
\U(f)\colon \U(X)\to \U(Y) \qquad \textnormal{by} \qquad \U(f)(C)= \min \{D\in\U(Y) \mid f(C)\subseteq D \},
\end{displaymath}
and
\begin{displaymath}
\F(f)\colon \F(X)\to \F(Y) \qquad \textnormal{by} \qquad \F(f)(C)= \max \{D\in\F(Y) \mid f(C)\subseteq D \}.
\end{displaymath}
If, in addition, $\U(X)\cap\F(X)=\varnothing$, we define 
\begin{displaymath}
\C(f)\colon \C(X)\to \C(Y) \qquad \textnormal{by} \qquad \C(f)(C)= 
\begin{cases}
\U(f)(C) & \textup{if } C\in\U(X), \\
\F(f)(C) & \textup{if } C\in\F(X).
\end{cases}
\end{displaymath}
\end{definition}

We will prove now that $\U(f)$ is well defined. Let $C\in\U(X)$. Then there exists $a\in \mxl[X]$ such that $C\subseteq U_a$. Let $b\in\mxl[Y]$ such that $f(a)\leq b$. Then $f(C)\subseteq U_{f(a)}\subseteq U_b$. Hence $\ub{f(C)}\neq \varnothing$. Then $\{D\in\U(Y) \mid f(C)\subseteq D \}$ is non-empty by \ref{lema-para_conexo_minimal}, and thus it has a minimum element by \ref{prop_conexo_minimal}.

In a similar way it can be proved that $\F(f)$ is well defined and thus $\C(f)$ is also well defined.

\medskip

Note that the map $\C(f)$ might not be well defined without the assumption $\U(X)\cap\F(X)=\varnothing$. Note also that if $\U(X)\cap\F(X)\neq\varnothing$ then either $X$ is not connected or it has minimum and maximum elements by \ref{prop_UF_disjuntos}.

\begin{rem} \label{rem_def_de_Uf_and_Ff}
Let $X$ and $Y$ be finite $T_0$--spaces and let $f\colon X\to Y$ be a continuous map. 

Observe that, if $C\in\U(X)$, then, by \ref{prop_conexo_minimal}, $\U(f)(C)$ is the connected component of $U_{\ub {f(C)}}$ that contains $f(C)$.

In a similar way, if $C\in\F(X)$, then $\F(f)(C)$ is the connected component of $F_{\lb {f(C)}}$ that contains $f(C)$.
\end{rem}

We will show in the following proposition that the maps defined in \ref{def_Uf_Ff_Cf} are continuous.

\begin{prop}
Let $X$ and $Y$ be finite $T_0$--spaces and let $f\colon X\to Y$ be a continuous map. Then the maps $\U(f)$ and $\F(f)$ are continuous. If, in addition, $\U(X)\cap\F(X)=\varnothing$ then the map $\C(f)\colon \C(X)\to \C(Y)$ is continuous.
\end{prop}

\begin{proof}
Let $C_1, C_2\in \C(X)$ such that $C_1\leq C_2$. 
\begin{itemize}
 \item If $C_1,C_2\in \U(X)$, then $C_1\subseteq C_2$. Hence $f(C_1)\subseteq f(C_2)\subseteq \U(f)(C_2)$. Then $\U(f)(C_1)\subseteq \U(f)(C_2)$ by definition of $\U(f)$. Thus, $\U(f)(C_1)\leq \U(f)(C_2)$.
 \item If $C_1,C_2\in \F(X)$, then $C_2\subseteq C_1$ and $f(C_2)\subseteq f(C_1)\subseteq \F(f)(C_1)$. Hence $\F(f)(C_2)\subseteq \F(f)(C_1)$ by definition of $\F(f)$. Thus, $\F(f)(C_1)\leq \F(f)(C_2)$.
 \item If $C_1\in \F(X)$ and $C_2\in \U(X)$, then $C_1\cap C_2\neq \varnothing$. Hence $\varnothing \neq f(C_1)\cap f(C_2)\subseteq \C(f)(C_1)\cap \C(f)(C_2)$. Thus, $\C(f)(C_1)\leq \C(f)(C_2)$.
\end{itemize}
The result follows.
\end{proof}

\begin{prop} \label{prop_Uf_homeo}
Let $X$ and $Y$ be finite $T_0$--spaces and let $f\colon X\to Y$ be a homeomorphism. Then 
\begin{enumerate}[(a)]
\item for all $C\in \U(X)$, $\U(f)(C)=f(C)$,
\item for all $C\in \F(X)$, $\F(f)(C)=f(C)$,
\item $\U(f)$ and $\F(f)$ are homeomorphisms.
\end{enumerate}
\end{prop}

\begin{proof}\ 

($a$) Since $f$ is a homeomorphism, $f(\mxl[X])=\mxl[Y]$ and $f(U_A)=U_{f(A)}$ for each non-empty subset $A\subseteq \mxl[X]$. Let $C\in \U(X)$. Then $C$ is a connected component of $U_{\ub C}$ and hence $f(C)$ is a connected component of $U_{f(\ub C)}$. Thus $\U(f)(C)=f(C)$. 

($b$) Analogous to ($a$).

($c$) Let $C\in \U(X)$. Then applying ($a$) we obtain that $\U(f^{-1})\U(f)(C)=\U(f^{-1})(f(C))=f^{-1}(f(C))=C$. Thus $\U(f^{-1})\U(f) = \id_{\U(X)}$. Similarly $\U(f)\U(f^{-1}) = \id_{\U(Y)}$.

The proof for $\F(f)$ follows similarly.
\end{proof}

In the following example we will see that the constructions $\U$, $\F$ and $\C$ are not functorial.

\begin{ex}
Let $X$ be the finite $T_0$--space whose underlying set is $\{0,1,2,3,4\}$ and whose topology is generated by the basis $\{\{0\},\{0,1\},\{2\},\{0,1,2,3\},\{0,1,2,4\}\}$. The Hase diagrams of $X$ and $\U(X)$ are shown below.
\smallskip
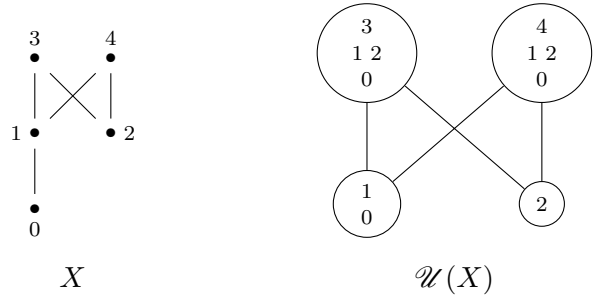
\begin{center}
\begin{tikzpicture}[baseline=-1.25cm]
\tikzstyle{every node}=[font=\scriptsize]
\draw (0,0) node(0){$\bullet$} node[below=1]{0}; 
\draw (0,1) node (1){$\bullet$} node[left=1]{1};
\draw (1,1) node (2){$\bullet$} node[right=1]{2};
\foreach \x in {3,4} \draw (\x-3,2) node(\x){$\bullet$} node[above=1]{\x}; 
  
\draw (0)--(1);
\foreach \x in {4,3} \draw (1)--(\x);
\foreach \x in {4,3} \draw (2)--(\x);
\draw (0.5,-0.9) node(X){\normalsize $X$};
\end{tikzpicture}
\hspace*{2cm}
\begin{tikzpicture}[x=2.3cm, y=2cm]
\path node (U3) at (0,1) [CC] {3\\1 2\\0}
 node (U4) at (1,1) [CC] {4\\1 2\\0}
 node (U1) at (0,0) [CC] {1\\0}
 node (U2) at (1,0) [CC] {2};

\draw (U1)--(U4)--(U2)--(U3)--(U1);
\draw (0.5,-0.5) node(UX){$\U(X)$};
\end{tikzpicture}
\end{center}

Let $f\colon X\to X$ be the map defined by
\begin{center}
\begin{tabular}{c | c c c c c}
$x$ & 0 & 1 & 2 & 3& 4\\[-2pt]
\hline
$f(x)$& 2 & 3 & 3 & 3& 3
\end{tabular}
\end{center}
Clearly $f$ is a continuous map. Let $g\colon X\to X$ be the constant map with value 0. Then $fg$ is the constant map with value 2. It follows that $\U(g)$ is the constant map with value $\{0,1\}$ and that $\U(fg)$ is the constant map with value $\{2\}$. Then 
\[\U(f)\U(g)(\{2\})=\U(f)(\{0,1\})= U_3 \neq \{2\} = \U(fg)(\{2\}).
\]
Hence, $\U(fg)\neq\U(f)\U(g)$. 

This implies that $\C(fg)\neq\C(f)\C(g)$, and the fact that $\F$ is not functorial can be deduced from this example considering opposite spaces.
\end{ex}

\begin{prop}
Let $X$, $Y$ and $Z$ be finite $T_0$--spaces and let $g\colon X\to Y$ and $f\colon Y\to Z$ be continuous maps. Then $\U(fg) \leq \U(f)\U(g)$ and $\F(fg) \geq \F(f)\F(g)$. If, in addition, $\U(X)\cap\F(X)=\varnothing$ and $\U(Y)\cap\F(Y)=\varnothing$ then $\C(fg) \simeq \C(f)\C(g)$.
\end{prop}

\begin{proof}
Let $C\in \U(X)$. Then
\[fg(C)\subseteq f(\U(g)(C))\subseteq \U(f)(\U(g)(C)).
\]
Hence, $\U(fg)(C)\subseteq\U(f)(\U(g)(C))$. Therefore, $\U(fg) \leq \U(f)\U(g)$. In a similar way it can be proved that $\F(fg) \geq \F(f)\F(g)$.

If, in addition, $\U(X)\cap\F(X)=\varnothing$ and $\U(Y)\cap\F(Y)=\varnothing$ then the maps $\C(fg)$ and $\C(f)\C(g)$ are defined. Since $\U(fg) \leq \U(f)\U(g)$ and $\F(fg) \geq \F(f)\F(g)$ we obtain that, for all $C\in\C(X)$, the elements $\C(fg)(C)$ and $\C(f)\C(g)(C)$ are comparable. Thus, the last assertion follows.
\end{proof}

\begin{prop} \label{prop_f_leq_g_Uf_leq_Ug}
Let $X$ and $Y$ be finite $T_0$--spaces and let $f,g\colon X\to Y$ be continuous maps. 
\begin{enumerate}[(a)]
\item If $f\leq g$, then $\U(f)\leq \U(g)$ and $\F(f)\leq \F(g)$.
\item If $f\leq g$ and $\U(X)\cap \F(X)=\varnothing$, then $\C(f)\leq \C(g)$.
\item If $f\simeq g$, then $\U(f)\simeq \U(g)$ and $\F(f)\simeq \F(g)$.
\item If $f\simeq g$ and $\U(X)\cap \F(X)=\varnothing$, then $\C(f)\simeq \C(g)$.
\end{enumerate}
\end{prop}

\begin{proof}

We will prove ($a$). Item ($b$) follows from ($a$), while items ($c$) and ($d$) follow from ($a$) and ($b$).

Let $C\in \U(X)$. If $x\in C$ then $f(x)\leq g(x)\in g(C)\subseteq \U(g)(C)$, and since $\U(g)(C)$ is an open subset of $Y$ it follows that $f(x)\in \U(g)(C)$. Therefore, $f(C)\subseteq \U(g)(C)$. Hence $\U(f)(C)\leq \U(g)(C)$.

Now let $C\in \F(X)$. If $x\in C$ then $g(x)\geq f(x)\in f(C)\subseteq \F(f)(C)$, and since $\F(f)(C)$ is a closed subset of $Y$ it follows that $g(x)\in \F(f)(C)$. Therefore, $g(C)\subseteq \F(f)(C)$. Hence $\F(f)(C)\leq \F(g)(C)$.
\end{proof} 

\begin{coro}
Let $X$ and $Y$ be finite $T_0$--spaces and let $f\colon X\to Y$ be a homotopy equivalence. Then $\U(f)$ and $\F(f)$ are homotopy equivalences. If, in addition, $\U(X)\cap\F(X)=\varnothing$ and $\U(Y)\cap\F(Y)=\varnothing$ then $\C(f)$ is a homotopy equivalence.
\end{coro}

\section{Fixed points}

The following proposition and its corollaries relate the $\C$ construction to the fixed point property.

\begin{prop}
\label{prop_pf_con_ppf_implica_f_pf_general}
Let $X$ be a finite $T_0$--space and let $f\colon X\to X$ be a continuous map. Let $C\in \C(X)$ be a subspace of $X$ with the fixed point property. 
\begin{enumerate}[(a)]
\item If $C\in \U(X)$ and $\U(f)(C)\leq C$ then $f$ has a fixed point.
\item If $C\in \F(X)$ and $\F(f)(C)\geq C$ then $f$ has a fixed point.
\item If $\U(X)\cap\F(X)=\varnothing$ and $\C(f)(C)=C$, then $f$ has a fixed point.
\end{enumerate}
\end{prop}

\begin{proof}
Under the assumptions of item ($a$) we obtain that $f(C)\subseteq \U(f)(C) \subseteq C$. Hence the map $f$ has a fixed point. Item ($b$) can be proved in a similar way and item ($c$) follows immediately from ($a$) and ($b$).
\end{proof}

\begin{coro}
Let $X$ be a finite $T_0$--space. 
\begin{itemize}
\item If the space $\U(X)$ (resp. $\F(X)$) has the fixed point property and every $C\in\U(X)$ (resp. every $C\in\F(X)$) has the fixed point property then $X$ has the fixed point property.
\item If $\U(X)\cap\F(X)=\varnothing$, the space $\C(X)$ has the fixed point property and every $C\in\C(X)$ has the fixed point property then $X$ has the fixed point property.
\end{itemize}
\end{coro}

\begin{proof}
We will prove the result in the case that $\U(X)$ has the fixed property. Let $f\colon X \to X$ be a continuous map. Then the map $\U(f)\colon\U(X)\to\U(X)$ has a fixed point. Thus $f$ has a fixed point by the previous proposition.
\end{proof}

\begin{coro}
Let $X$ be a finite $T_0$--space  and let $f\colon X\to X$ be a continuous map.
\begin{enumerate}[(a)]
\item If there exist $a\in \mxl[X]$ and $C\in \U(X)$ such that $U_a\geq C$ and $\U(f)(C) = U_a$, then $f$ has a fixed point.
\item If there exist $a\in \mnl[X]$ and $C\in \F(X)$ such that $F_a\leq C$ and $\F(f)(C) = F_a$, then $f$ has a fixed point.
\end{enumerate}
\end{coro}

\begin{proof}
From the hypotheses of item ($a$) we obtain that $\U(f)(U_a)\geq \U(f)(C) = U_a$. And since $U_a\in \mxl[\C(X)]$, $\U(f)(U_a) = U_a$. Then $f$ has a fixed point by proposition \ref{prop_pf_con_ppf_implica_f_pf_general}.

The second item follows in a similar way.
\end{proof}

\begin{prop}
Let $X$ be a finite $T_0$--space and let $f\colon X\to X$ be a continuous map. If $f$ has a fixed point, then the maps 
$\U(f)$ and $\F(f)$ have fixed points.
\end{prop}

\begin{proof}
Let $x\in X$ be such that $f(x) = x$. Let $C_0=\min\{C\in\U(X) \mid x\in C\}$. Note that this minimum element exists by proposition \ref{prop_conexo_minimal}. It follows that $x = f(x)\in f(C_0)\subseteq \U(f)(C_0)$. Then $C_0\leq \U(f)(C_0)$. Hence, $\U(f)$ has a fixed point.
 
In a similar way it can be proved that the map $\F(f)$ has a fixed point.
\end{proof}

The following example shows that the converse of the previous proposition does not hold.

\begin{ex} \label{ex_2}
Let $X$ be the finite $T_0$--space whose underlying set is $\{0,1,2,3,4,5\}$ and whose topology is generated by the basis $\{\{0\},\{1\},\{0,1,2\},\{0,1,3\},\{0,1,2,3,4\},\{0,1,2,3,5\}\}$. The Hase diagrams of $X$ and $\U(X)$ are shown below.

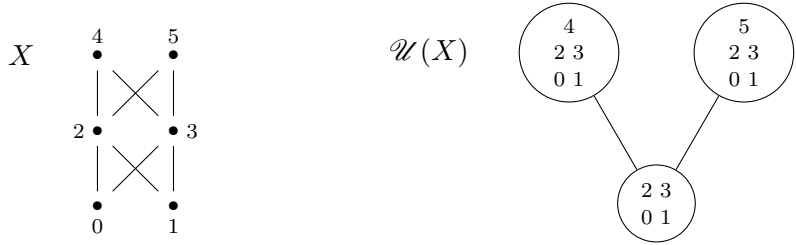
\begin{center}
\begin{tikzpicture}
\tikzstyle{every node}=[font=\scriptsize]
\foreach \x in {0,1} \draw (\x,0) node(\x){$\bullet$} node[below=1]{\x}; 
\draw (0,1) node (2){$\bullet$} node[left=1]{2};
\draw (1,1) node (3){$\bullet$} node[right=1]{3};
\foreach \x in {4,5} \draw (\x-4,2) node(\x){$\bullet$} node[above=1]{\x}; 
  
\foreach \x in {2,3} \draw (0)--(\x);
\foreach \x in {2,3} \draw (1)--(\x);
\foreach \x in {4,5} \draw (2)--(\x);
\foreach \x in {4,5} \draw (3)--(\x);
\draw (-1,2) node(X){\normalsize $X$};
\end{tikzpicture}
\hspace*{2cm}
\begin{tikzpicture}[x=2.3cm, y=2cm]
\path node (U4) at (0,1) [CC] {4\\2 3\\0 1}
 node (U5) at (1,1) [CC] {5\\2 3\\ 0 1}
 node (U45) at (0.5,0) [CC] {2 3\\0 1};
 
\draw [-] (U45) to (U4); 
\draw [-] (U45) to (U5);
\draw (-0.8,1) node(UX){$\U(X)$};
\end{tikzpicture}
\end{center}

Let $f\colon X\to X$ be the continuous map defined by
\begin{center}
\begin{tabular}{c | c c c c c c }
$x$ & 0 & 1 & 2 & 3& 4 & 5 \\[-2pt]
\hline
$f(x)$& 1 & 0 & 3 & 2& 5 & 4 
\end{tabular}
\end{center}
Clearly $\U(X)$ has the fixed point property and hence the map $\U(f)$ has a fixed point. In a similar way, $\F(f)$ has a fixed point. However, the map $f$ does not have fixed points.
\end{ex}

\section{Examples}

In this section we give several examples of application which show that the tools of section \ref{section-C_constr} are useful to study the fixed point property of finite posets.

First we give the following lemmas.

\begin{lemma}
\label{lema_para_Q4_Q5}
 Let $X$ the finite $T_0$--space given by the following Hasse diagram.
 
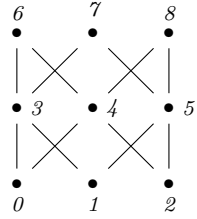
\[
 \begin{tikzpicture}
 \tikzstyle{every node}=[font=\scriptsize]
  
  \foreach \x in {0,1,2} \draw (\x,0) node(\x){$\bullet$} node[below=1]{\x}; 
  \foreach \x in {3,4,5} \draw (\x-3,1) node(\x){$\bullet$} node[right=1]{\x}; 
  \foreach \x in {6,7,8} \draw (\x-6,2) node(\x){$\bullet$} node[above=1]{\x}; 
  
  \foreach \x in {3,4} \draw (0)--(\x);
  \foreach \x in {3,5} \draw (1)--(\x);
  \foreach \x in {4,5} \draw (2)--(\x);
 
  \foreach \x in {6,7} \draw (3)--(\x);
  \foreach \x in {6,8} \draw (4)--(\x);
  \foreach \x in {7,8} \draw (5)--(\x);
 \end{tikzpicture}
 \]
 
If $f\colon X\to X$ is a continuous map without fixed points, then $f(\{3,4,5\})=\{3,4,5\}$.
\end{lemma}

\begin{proof}
We will prove first that $f(\{3,4,5\})\subseteq X-\mxl[X]$. Let $x\in \{3,4,5\}$ and suppose that $f(x)\in \mxl[X]$. Then $f(F_x)=\{f(x)\}$ and hence there exist $a,b\in \mxl[X]$ such that $a\neq b$ and $f(a)=f(b)=f(x)$. Observe that $\{0,1,2,3,4,5\}\subseteq U_a \cup U_b$. Hence 
\[f(\widehat{U}_{f(x)})\subseteq \{0,1,2,3,4,5\}\subseteq f(U_a \cup U_b) \subseteq U_{f(x)}.
 \] 
Since $\widehat{U}_{f(x)}$ is contractible, $f$ has a fixed point by \ref{Uhat} which entails a contradiction. 

Since $f(\{3,4,5\})\subseteq X-\mxl[X]$, it follows that $f(X-\mxl[X])\subseteq X-\mxl[X]$. And since $f$ does not have fixed points, $f$ is bijective in this subspace by \ref{prop_crown}. In particular $f(\{3,4,5\})=\{3,4,5\}$.
\end{proof}

\begin{lemma} \label{lema_para_Q7}
  Let $X$ the finite $T_0$--space given by the following Hasse diagram.
 
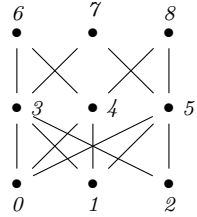
\[
 \begin{tikzpicture}
 \tikzstyle{every node}=[font=\scriptsize]
  \foreach \x in {0,1,2} \draw (\x,0) node(\x){$\bullet$} node[below=1]{\x}; 
  \foreach \x in {3,4,5} \draw (\x-3,1) node(\x){$\bullet$} node[right=1]{\x}; 
  \foreach \x in {6,7,8} \draw (\x-6,2) node(\x){$\bullet$} node[above=1]{\x}; 
  
   \foreach \x in {3,4,5} \draw (0)--(\x);
   \foreach \x in {3,4,5} \draw (1)--(\x);
   \foreach \x in {3,5} \draw (2)--(\x);
 
   \foreach \x in {6,7} \draw (3)--(\x);
   \foreach \x in {6,8} \draw (4)--(\x);
   \foreach \x in {7,8} \draw (5)--(\x);
 \end{tikzpicture}
 \]
 
 Let $f\colon X\to X$ a continuous map without fixed points. Then $f(\{3,4,5\}) = \{3,4,5\}$.
\end{lemma}

\begin{proof}
Let $x\in \{0,1\}$. Since $f$ does not have fixed points, $f(x)$ is not comparable with $x$. Hence $f(x)\in \{0,1,2\}-\{x\}$. In addition, note that $\widehat F_2\subseteq F_x$. Then, if $f(x) = 2$, it follows that $f(\widehat F_2)\subseteq f(F_x) \subseteq F_2$, with $\widehat F_2$ contractible. Hence by \ref{Uhat}, $f$ has a fixed point, which entails a contradiction. Then $f(0)=1$ and $f(1)=0$. It follows that $f(F_{0,1})\subseteq F_{0,1}$. In addition, as $f$ does not have fixed points, the restriction $f|\colon F_{0,1} \to F_{0,1}$ is bijective by \ref{prop_crown}. In particular, $f(\{3,4,5\}) = \{3,4,5\}$.
\end{proof}

Now we are ready to give the examples.

\begin{figure}
\begin{subfigure}[b]{0.25\textwidth}
\centering
\begin{tikzpicture}
\tikzstyle{every node}=[font=\scriptsize]
\foreach \x in {0,1,2} \draw (\x,0) node(\x){$\bullet$} node[below=1]{\x}; 
\foreach \x in {3,4,5} \draw (\x-3,1) node(\x){$\bullet$} node[right=1]{\x}; 
\foreach \x in {6,7} \draw (\x-5.5,2) node(\x){$\bullet$} node[right=1]{\x}; 
\foreach \x in {8,9,10} \draw (\x-8,3) node(\x){$\bullet$} node[above=1]{\x}; 
  
\foreach \x in {3,4} \draw (0)--(\x);
\foreach \x in {3,5} \draw (1)--(\x);
\foreach \x in {4,5} \draw (2)--(\x);
\foreach \x in {8,10} \draw (3)--(\x);
\foreach \x in {6,7} \draw (4)--(\x);
\foreach \x in {6,7} \draw (5)--(\x);
\foreach \x in {8,9} \draw (6)--(\x);
\foreach \x in {9,10} \draw (7)--(\x);
\end{tikzpicture}
\caption{$P^{3323}$}
\label{fig-Q8}
\end{subfigure}
\begin{subfigure}[b]{0.35\textwidth}
\centering
\begin{tikzpicture}
\tikzstyle{every node}=[font=\scriptsize]
\foreach \x in {0,1,2} \draw (\x,-0.2) node(\x){$\bullet$} node[below=1]{\x}; 
\foreach \x in {3,4,5,6,7} \draw (\x-4,1) node(\x){$\bullet$} node[right=1]{\x}; 
\foreach \x in {8,9,10} \draw (\x-8,2.2) node(\x){$\bullet$} node[above=1]{\x}; 
  
\foreach \x in {3,4,5} \draw (0)--(\x);
\foreach \x in {3,4,6,7} \draw (1)--(\x);
\foreach \x in {5,6,7} \draw (2)--(\x);

\foreach \x in {3,4,6,7} \draw (8)--(\x);
\foreach \x in {3,5,7} \draw (9)--(\x);
\foreach \x in {4,5,6} \draw (10)--(\x);
\end{tikzpicture}
\caption{$P^{343}_1$}
\label{fig-Q6}
\end{subfigure}
\begin{subfigure}[b]{0.35\textwidth}
\begin{tikzpicture}
\tikzstyle{every node}=[font=\scriptsize]
\foreach \x in {0,1,2} \draw (\x,-0.2) node(\x){$\bullet$} node[below=1]{\x}; 
\foreach \x in {3,4,5,6,7} \draw (\x-4,1) node(\x){$\bullet$} node[right=1]{\x}; 
\foreach \x in {8,9,10} \draw (\x-8,2.2) node(\x){$\bullet$} node[above=1]{\x}; 
  
\foreach \x in {3,4,5} \draw (0)--(\x);
\foreach \x in {3,4,6,7} \draw (1)--(\x);
\foreach \x in {5,6,7} \draw (2)--(\x);

\foreach \x in {3,4,6} \draw (8)--(\x);
\foreach \x in {3,5,6,7} \draw (9)--(\x);
\foreach \x in {4,5,7} \draw (10)--(\x);
\end{tikzpicture}
\caption{$P^{343}_2$}
\label{fig-Q7}
\end{subfigure}
\caption{}
\end{figure}

\begin{ex}
\label{ejm_Q8}
We will prove that the space $P^{3323}$ of \cite[Fig.1]{Schroder1993fpp11} has the fixed point property. Its Hasse diagram is given in figure \ref{fig-Q8}. The Hasse diagrams of $\U(P^{3323})$ and $\F(P^{3323})$ are shown in figure \ref{fig-CQ8} and the shaded elements of these diagrams are contractible subspaces of $P^{3323}$.

\begin{figure}
\begin{subfigure}[b]{0.48\textwidth}
\centering
\begin{tikzpicture}[x=2.3cm, y=2cm]
\path node[CCC] (U8) at (0,6) { 8\\ 6\\ 3 4 5\\ 0 1 2}
node[CCC] (U9) at (1,6) { 9\\ 6 7\\ 4 5\\ 0 1 2}
node (U10) at (2,6) [CCC] { 10\\ 7\\ 3 4 5\\ 0 1 2}

node (U89) at (0,5) [CCC] { 6\\ 4 5\\ 0 1 2} 
node (U810) at (1,5) [CC] { 3 4 5\\ 0 1 2} 
node (U910) at (2,5) [CCC] { 7\\ 4 5\\ 0 1 2}

node (U8910) at (1,4) [CCC] { 4 5\\ 0 1 2};

\draw [-] (U8) to (U810) to (U10) to (U910) to (U9) to (U89) to (U8);
\draw [-] (U8910) to (U89);
\draw [-] (U8910) to (U810);
\draw [-] (U8910) to (U910);
\end{tikzpicture}
\caption{$\U(P^{3323})$}
\label{fig-UQ8}
\end{subfigure}
\begin{subfigure}[b]{0.48\textwidth}
\centering
\begin{tikzpicture}[x=2.3cm, y=2cm]
\path node (F012) at (1,3) [CCC] { 8 9 10\\ 6 7}

node (F01) at (0,2) [CC] { 8 9 10\\ 6 7\\ 3} 
node (F02) at (1,2) [CCC] { 8 9 10\\ 6 7\\ 4} 
node (F12) at (2,2) [CCC] { 8 9 10\\ 6 7\\ 5}

node (F0) at (0,1) [CCC] { 8 9 10\\ 6 7\\ 3 4\\ 0} 
node (F1) at (1,1) [CCC] { 8 9 10\\ 6 7\\ 3 5\\ 1} 
node (F2) at (2,1) [CCC] { 8 9 10\\ 6 7\\ 4 5\\ 2};

\draw[-] (F012) to (F01);
\draw[-] (F012) to (F02);
\draw[-] (F012) to (F12);
\draw[-] (F0) to (F02) to (F2) to (F12) to (F1) to (F01) to (F0);

\end{tikzpicture}
\caption{$\F(P^{3323})$}
\label{fig-FQ8}
\end{subfigure}
\caption{\!}
\label{fig-CQ8}
\end{figure}

Suppose that there exists a continuous map $f\colon P^{3323}\to P^{3323}$ without fixed points. 

Let $C_1=\{0,1,2,3,4,5\}\in \U(X)$. Since the poset $\U(P^{3323})$ has a minimum element, the map $\U(f)\colon \U(P^{3323})\to \U(P^{3323})$ has a fixed point, which must be $C_1$ by \ref{prop_pf_con_ppf_implica_f_pf_general}. Thus, $f(C_1)\subseteq \U(f)(C_1)=C_1$.

Let $C_2=\{3,6,7,8,9,10\}\in \F(P^{3323})$. Proceeding as in the previous paragraph we obtain that $C_2$ is a fixed point of $\F(f)$. Thus, $f(C_2)\subseteq \F(f)(C_2)=C_2$.

Hence, $f(C_1\cap C_2)\subseteq C_1\cap C_2$ and hence $3$ is a fixed point of $f$, which entails a contradiction.

Therefore, the space $P^{3323}$ has the fixed point property.
\end{ex}

\begin{ex}
We will prove that the space $P^{343}_1$ given in \cite[Fig.1]{Schroder1993fpp11} has the fixed point property. Its Hasse diagram is given in figure \ref{fig-Q6}. The Hasse diagram of $\C(P^{343}_1)$ is shown in figure \ref{fig-CQ6}.

\begin{figure}
\begin{subfigure}[c]{0.48\textwidth}
\centering
\begin{tikzpicture}[x=2.3cm, y=2cm]
\path node[CCC] (U8) at (0,6)  { 8\\ 3 4 6 7\\ 0 1 2}
node (U9) at (1,6) [CCC] { 9\\ 3 5 7\\ 0 1 2}
node (U10) at (2,6) [CCC] { 10\\ 4 5 6\\ 0 1 2}

node (U89) at (0,5) [CCC] { 3 7\\ 0 1 2} 
node (U810) at (1,5) [CCC] { 4 6\\ 0 1 2} 
node (U910a) at (2,5) [CCC] { 5\\ 0 2}

node (U8910a) at (0.2,4) [CCC] { 0}
node (U8910b) at (1,4) [CCC] { 1}
node (U8910c) at (1.8,4) [CCC] { 2};

\draw [-] (U8) to (U810) to (U10) to (U910a) to (U9) to (U89) to (U8);

\foreach \x in {U89, U810, U910a} \draw (U8910a)--(\x);
\foreach \x in {U89, U810} \draw (U8910b)--(\x);
\foreach \x in {U89, U910a, U810} \draw (U8910c)--(\x);

\path node (F012a) at (0.2,3) [CCC] { 8}
node (F012b) at (1,3) [CCC] { 9}
node (F012c) at (1.8,3) [CCC] { 10}

node (F01) at (0,2) [CCC] { 8 9 10\\ 3 4} 
node (F02a) at (1,2) [CCC] { 9 10\\ 5} 
node (F12) at (2,2) [CCC] { 8 9 10\\ 6 7}

node (F0) at (0,1) [CCC] { 8 9 10\\ 3 4 5\\ 0} 
node (F1) at (1,1) [CCC] { 8 9 10\\ 3 4 6 7\\ 1} 
node (F2) at (2,1) [CCC] { 8 9 10\\ 5 6 7\\ 2};

\foreach \x in {F01, F12} \draw (F012a)--(\x);
\foreach \x in {F01, F02a, F12} \draw (F012b)--(\x);
\foreach \x in {F01, F02a, F12} \draw (F012c)--(\x);
\draw[-] (F0) to (F02a) to (F2) to (F12) to (F1) to (F01) to (F0);

\draw[-, red] (U810) to (F01) to (U89) to (F12) to (U810);
\draw[-, red] (F02a) to (U910a);
\draw[-, blue, bend left] (U8) to (F012a);
\draw[-, blue, bend left] (U9) to (F012b);
\draw[-, blue, bend left] (U10) to (F012c);

\draw[-, blue, bend left] (F0) to (U8910a);
\draw[-, blue, bend left] (F1) to (U8910b);
\draw[-, blue, bend left] (F2) to (U8910c);
\end{tikzpicture}
\caption{$\C(P^{343}_1)$}
\label{fig-CQ6}
\end{subfigure}
\begin{subfigure}[c]{0.48\textwidth}
\centering
\begin{tikzpicture}[x=2.3cm, y=2cm]
\path node (U8) at (0,6) [CCC] { 8\\ 3 4 6\\ 0 1 2}
node (U9) at (1,6) [CCC] { 9 \\ 3 5 6 7\\ 0 1 2}
node (U10) at (2,6) [CCC] { 10\\ 4 5 7\\ 0 1 2}

node (U89) at (0,5) [CCC, label=left:$A$] { 3 6\\ 0 1 2} 
node (U810a) at (1,5) [CCC, label=left:$B$] { 4\\ 0 1} 
node (U910) at (2,5) [CCC, label=left:$C$] { 5 7\\ 0 1 2}

node (U8910a) at (0.2,4) [CCC] { 0}
node (U8910b) at (1,4) [CCC] { 1}
node (U8910c) at (1.8,4) [CCC] { 2};

\draw [-] (U8) to (U810a) to (U10) to (U910) to (U9) to (U89) to (U8);

\foreach \x in {U89, U810a, U910} \draw (U8910a)--(\x);
\foreach \x in {U89, U810a, U910} \draw (U8910b)--(\x);
\foreach \x in {U89, U910} \draw (U8910c)--(\x);

\path node (F012a) at (0.2,3) [CCC] {8}
node (F012b) at (1,3) [CCC] {9}
node (F012c) at (1.8,3) [CCC] {10}

node (F01) at (0,2) [shape=circle, draw, fill=black!10, label=right:$A'$, align=center, font=\scriptsize] {8 9 10\\3 4} 
node (F02a) at (1,2) [shape=circle, draw, fill=black!10, label=right:$B'$, align=center, font=\scriptsize] {9 10\\5} 
node (F12) at (2,2) [shape=circle, draw, fill=black!10, label=right:$C'$,  align=center, font=\scriptsize] {8 9 10\\6 7}

node (F0) at (0,1) [CCC] {8 9 10\\3 4 5\\0} 
node (F1) at (1,1) [CCC] {8 9 10\\3 4 6 7\\1} 
node (F2) at (2,1) [CCC] {8 9 10\\5 6 7\\2};

\foreach \x in {F01, F12} \draw (F012a)--(\x);
\foreach \x in {F01, F02a, F12} \draw (F012b)--(\x);
\foreach \x in {F01, F02a, F12} \draw (F012c)--(\x);
\draw[-] (F0) to (F02a) to (F2) to (F12) to (F1) to (F01) to (F0);

\draw[-, red] (U810a) to (F01) to (U89) to (F12) to (U910) to (F02a);
\draw[-, blue, bend left] (U8) to (F012a);
\draw[-, blue, bend left] (U9) to (F012b);
\draw[-, blue, bend left] (U10) to (F012c);

\draw[-, blue, bend left] (F0) to (U8910a);
\draw[-, blue, bend left] (F1) to (U8910b);
\draw[-, blue, bend left] (F2) to (U8910c);
\end{tikzpicture}
\caption{$\C(P^{343}_2)$}
\label{fig-CQ7}
\end{subfigure}
\caption{}
\end{figure}

We will prove first that $\C(P^{343}_1)$ is homotopy equivalent to $P^{3323}$. Note that the minimal elements of $\U(P^{343}_1)$ are down beat points of $\C(P^{343}_1)$ and that the maximal elements of $\F(P^{343}_1)$ are up beat points of $\C(P^{343}_1)$. Then $\C(P^{343}_1)$ is homotopy equivalent to the finite $T_0$--space that is given by the Hasse diagram of figure \ref{fig-step-1}. Now observe that the element labeled $a$ in this figure is a down beat point, and by removing it we obtain the space given in figure \ref{fig-step-2}, which is homeomorphic to $P^{3323}$. Therefore $\C(P^{343}_1)$ is homotopy equivalent to $P^{3323}$.

\begin{figure}
\begin{subfigure}[c]{0.4\textwidth}
\centering
\begin{tikzpicture}
 \tikzstyle{every node}=[font=\scriptsize]
  \foreach \x in {0,1,2} \draw (\x,0) node(\x){$\bullet$}; 
  \foreach \x in {3,4,5} \draw (\x-3,1) node(\x){$\bullet$}; 

  \foreach \x in {9,10} \draw (\x-9,3) node(\x){$\bullet$};
  \draw (2,3) node(11){$\bullet$} node[right=1]{{\normalsize $a$}};
  \foreach \x in {12,13,14} \draw (\x-12,4) node(\x){$\bullet$}; 

  \foreach \x in {3,4} \draw (0)--(\x);
  \foreach \x in {3,5} \draw (1)--(\x);
  \foreach \x in {4,5} \draw (2)--(\x);

  \foreach \x in {12,13} \draw (9)--(\x);
  \foreach \x in {12,14} \draw (10)--(\x);
  \foreach \x in {13,14} \draw (11)--(\x);

  \draw[-, red] (3) to (9) to (5) to (10) to (3);
  \draw[-, red] (4) to (11);
\end{tikzpicture}
\caption{}
\label{fig-step-1}
\end{subfigure}
\begin{subfigure}[c]{0.4\textwidth}
\centering
\begin{tikzpicture}
  \tikzstyle{every node}=[font=\scriptsize]
  \foreach \x in {0,1,2} \draw (\x,0) node(\x){$\bullet$}; 
  \foreach \x in {3,4,5} \draw (\x-3,1) node(\x){$\bullet$}; 

  \foreach \x in {9,10} \draw (\x-9,3) node(\x){$\bullet$};
  \foreach \x in {12,13,14} \draw (\x-12,4) node(\x){$\bullet$}; 

  \foreach \x in {3,4} \draw (0)--(\x);
  \foreach \x in {3,5} \draw (1)--(\x);
  \foreach \x in {4,5} \draw (2)--(\x);

  \foreach \x in {12,13} \draw (9)--(\x);
  \foreach \x in {12,14} \draw (10)--(\x);

  \draw[-, red] (3) to (9) to (5) to (10) to (3);
  \draw[-,bend right] (4) to (13);
  \draw[-,bend right] (4) to (14);
\end{tikzpicture}
\caption{}
\label{fig-step-2}
\end{subfigure}
\caption{}
\end{figure}

Now, let $f\colon P^{343}_1\to P^{343}_1$ be a continuous map. In example \ref{ejm_Q8} we proved that $P^{3323}$ has the fixed point property. Hence, by \ref{prop_XheY_ppf}, $\C(P^{343}_1)$ has the fixed point property. Thus, the map $\C(f)$ has a fixed point. And since all the elements of $\C(P^{343}_1)$ are contractible subspaces of $P^{343}_1$, applying proposition \ref{prop_pf_con_ppf_implica_f_pf_general}, we obtain that $f$ has a fixed point. Therefore $P^{343}_1$ has the fixed point property.
\end{ex}

\begin{ex}
We will prove that the space $P^{343}_2$ given in \cite[Fig.1]{Schroder1993fpp11} has the fixed point property. Its Hasse diagram is given in figure \ref{fig-Q7}. The Hasse diagram of $\C(P^{343}_2)$ is shown in figure \ref{fig-CQ7}. Let $A$, $B$, $C$, $A'$, $B'$ and $C'$ be the sets indicated in that figure.

Suppose that $f\colon P^{343}_2\to P^{343}_2$ is a continuous map without fixed points. Since all the elements of $\U(P^{343}_2)$ are contractible subspaces of $P^{343}_2$, then the map $\U(f)$ does not have fixed points by \ref{prop_pf_con_ppf_implica_f_pf_general}. By lemma \ref{lema_para_Q7}, $\U(f)(\{A,B,C\})\subseteq \{A,B,C\}$. In a similar way, $\F(f)(\{A',B',C'\})\subseteq \{A',B',C'\}$. Consider the subspace $M = \{A,B,C,A',B',C'\}\subseteq \C(P^{343}_2)$. It follows that $\C(f)(M)\subseteq M$. Since $M$ is contractible, the map $\C(f)$ has fixed points and then, from \ref{prop_pf_con_ppf_implica_f_pf_general} we obtain that $f$ has fixed points, which entails a contradiction.

Therefore, $P^{343}_2$ has the fixed point property.
\end{ex}

\begin{ex}
Let $n,k\in\N$ such that $n\geq 4$ and $2\leq k \leq n-1$. Let $X_{n,k}$ the finite $T_0$--space whose underlying set is $\{a_1,a_2,a_3,b_1,b_2,\ldots,b_n,c_1,c_2,\ldots,c_n\}$ and whose minimal open sets are 
\begin{align*}
U_{a_1}&=\{a_1\}\cup \{b_j\mid 1\leq j\leq n-1\}\cup \{c_j\mid 1\leq j\leq n\},\\
U_{a_2}&=\{a_2\}\cup \{b_j\mid 1\leq j\leq n \land j\neq n-1 \}\cup \{c_j\mid 1\leq j\leq n\},\\
U_{a_3}&=\{a_3\}\cup \{b_j\mid k\leq j\leq n\}\cup \{c_j\mid k-1\leq j\leq n\},\\
U_{b_1}&=\{b_1,c_1,c_2\},\\
U_{b_j}&=\{b_j,c_{j-1},c_{j+1}\}, \textup{ for $2\leq j\leq n-1$},\\ 
U_{b_n}&=\{b_n,c_{n-1},c_n\},\\
U_{c_j}&=\{c_j\}, \textup{ for $1\leq j\leq n$}.
\end{align*}

The Hasse diagram of $X_{n,k}$ is
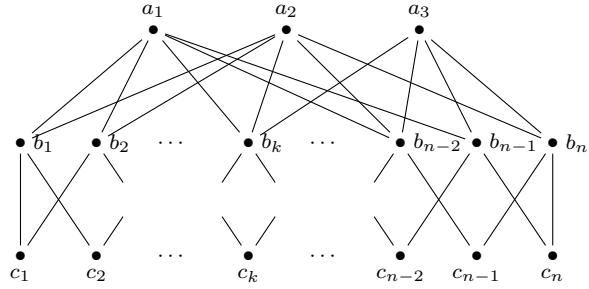
\begin{center}
\begin{tikzpicture}[y=1.5cm]
\tikzstyle{every node}=[font=\scriptsize]
\foreach \x in {1,2} \draw (\x,0) node(c\x)[inner sep=2pt]{$\bullet$} node[below=1]{$c_\x$}; 
\foreach \x in {3,5} \draw (\x,0) node(c\x)[inner sep=2pt]{$\cdots$};
\foreach \x in {4} \draw (\x,0) node(c\x)[inner sep=2pt]{$\bullet$} node[below=1]{$c_k$}; 
\foreach \x in {6} \draw (\x,0) node(c\x)[inner sep=2pt]{$\bullet$} node[below=1]{$c_{n-2}$}; 
\foreach \x in {7} \draw (\x,0) node(c\x)[inner sep=2pt]{$\bullet$} node[below=1]{$c_{n-1}$}; 
\foreach \x in {8} \draw (\x,0) node(c\x)[inner sep=2pt]{$\bullet$} node[below=1]{$c_n$}; 

\foreach \x in {1,2} \draw (\x,1) node(b\x)[inner sep=2pt]{$\bullet$} node[right=1]{$b_\x$}; 
\foreach \x in {3,5} \draw (\x,1) node(b\x)[inner sep=2pt]{$\cdots$};
\foreach \x in {4} \draw (\x,1) node(b\x)[inner sep=2pt]{$\bullet$} node[right=1]{$b_k$}; 
\foreach \x in {6} \draw (\x,1) node(b\x)[inner sep=2pt]{$\bullet$} node[right=1]{$b_{n-2}$}; 
\foreach \x in {7} \draw (\x,1) node(b\x)[inner sep=2pt]{$\bullet$} node[right=1]{$b_{n-1}$}; 
\foreach \x in {8} \draw (\x,1) node(b\x)[inner sep=2pt]{$\bullet$} node[right=1]{$b_n$}; 

\foreach \x in {1,2,3} \draw (1.75*\x+1,2) node(a\x)[inner sep=2pt]{$\bullet$} node[above=1]{$a_\x$}; 
  
\foreach \x in {1,2,4,6,7} \draw (a1)--(b\x);
\foreach \x in {1,2,4,6,8} \draw (a2)--(b\x);
\foreach \x in {4,6,7,8} \draw (a3)--(b\x);
\foreach \x in {1,2} \draw (b1)--(c\x);
\foreach \x in {1} \draw (b2)--(c\x);
\draw[shorten >=1cm,shorten <=0cm] (b2)--(c3);
\draw[shorten >=1cm,shorten <=0cm] (c2)--(b3);
\foreach \x in {3,5} \draw[shorten >=1cm,shorten <=0cm] (b4)--(c\x);
\foreach \x in {3,5} \draw[shorten >=1cm,shorten <=0cm] (c4)--(b\x);
\draw[shorten >=1cm,shorten <=0cm] (b6)--(c5);
\draw[shorten >=1cm,shorten <=0cm] (c6)--(b5);
\foreach \x in {7} \draw (b6)--(c\x);
\foreach \x in {6,8} \draw (b7)--(c\x);
\foreach \x in {7,8} \draw (b8)--(c\x);
\end{tikzpicture}
\end{center}

Observe that the posets $X_{4,3}$ and $X_{4,2}$ are isomorphic to the posets $P^{443}_4$ and $P^{443}_5$ of \cite[Fig.1]{Schroder1993fpp11}.

We will prove that the spaces $X_{n,k}$ have the fixed point property for all $n,k$ as above. To this end we will compute $\U(X_{n,k})$. Let
\begin{align*}
A &=\{b_j\mid 1\leq j \leq n-2\} \cup \{c_j\mid 1\leq j \leq n-1\}, \\
B &=\{b_j \mid j\geq k \ \land\ j\not\equiv n \modu{2} \} \cup \{c_j \mid j\geq k-1 \ \land\ j\equiv n \modu{2} \}, \\
C &=\{b_j \mid j\geq k \ \land\ j\equiv n \modu{2} \} \cup \{c_j \mid j\geq k-1 \ \land\ j\not\equiv n \modu{2} \}\cup\{c_n\}, \\
D &=\{b_j \mid k\leq j\geq n-3 \ \land\ j\not\equiv n \modu{2} \} \cup \{c_j \mid k-1\leq j \leq n-2 \ \land\ j\equiv n \modu{2} \}, \\
E &=\{b_j \mid k\leq j\leq n-2 \ \land\ j\equiv n \modu{2} \} \cup \{c_j \mid k-1\leq j\leq n-1 \ \land\ j\not\equiv n \modu{2} \}, \\
F &= \{c_n\}, \textnormal{ and } \\
S &= X_{n,k}-\mxl[X_{n,k}].
\end{align*}
Observe that the connected components of $U_{a_1a_2}$ are $A$ and $F$, the connected components of $U_{a_1a_3}$ are $B$ and $E$, the connected components of $U_{a_2a_3}$ are $C$ and $D$ and that the connected components of $U_{a_1a_2a_3}$ are $D$, $E$ and $F$. Hence $\U(X_{n,k})$ is the finite space given by the following Hasse diagram 
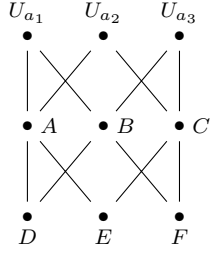
\begin{center}
\begin{tikzpicture}[x=1cm, y=1.2cm]
\tikzstyle{every node}=[font=\scriptsize]
\draw (0,2) node (Ua1){$\bullet$} node[above=1]{$U_{a_1}$};
\draw (1,2) node (Ua2){$\bullet$} node[above=1]{$U_{a_2}$};
\draw (2,2) node (Ua3){$\bullet$} node[above=1]{$U_{a_3}$};
 
\draw (0,1) node (A){$\bullet$} node[right=1]{$A$};
\draw (1,1) node (B){$\bullet$} node[right=1]{$B$};
\draw (2,1) node (C){$\bullet$} node[right=1]{$C$};

\draw (0,0) node (D){$\bullet$} node[below=1]{$D$};
\draw (1,0) node (E){$\bullet$} node[below=1]{$E$};
\draw (2,0) node (F){$\bullet$} node[below=1]{$F$};
 
\draw [-] (Ua1) to (B) to (Ua3) to (C) to (Ua2) to (A) to (Ua1);
\draw [-] (A) to (E) to (C) to (F) to (B) to (D) to (A);
\end{tikzpicture}
\end{center}

Suppose that $f\colon X_{n,k}\to X_{n,k}$ is a continuous map without fixed points. Since all the elements of $\U(X_{n,k})$ are contractible subspaces of $X_{n,k}$, then $\U(f)$ does not have fixed points by \ref{prop_pf_con_ppf_implica_f_pf_general}. Thus, by \ref{lema_para_Q4_Q5}, $\U(f)(\{A,B,C\})= \{A,B,C\}$. Since $S=A\cup B\cup C$, we have that $f(S)=f(A\cup B\cup C) = f(A)\cup f(B)\cup f(C) \subseteq A\cup B\cup C = S$. 

We will prove now that the restriction $f|\colon S\to S$ is not bijective. Note that 
\begin{align*}
\#A = 2n-3 > n-k +2 \geq \#B
\end{align*}
where the first inequality holds since $n+k >5$. Thus, the sets $A$, $B$ and $C$ do not all have the same cardinality and since $\U(f)$ does not have fixed points and $\U(f)(\{A,B,C\})= \{A,B,C\}$ it follows that there exists $T\in \{A,B,C\}$ such that $\# \U(f)(T)< \# T$. Hence $\# f(T)< \# T$ and then $f$ is not bijective.

Finally, since $f|\colon S\to S$ is not bijective, it has a fixed point by \ref{prop_crown} which entails a contradiction. 
\end{ex}

\section{Beat points}

\begin{rem}
Note that the space $X$ of example \ref{ex_2} is not contractible but $\U(X)$ is. Thus $X$ and $\U(X)$ do not have the same homotopy type. 
\end{rem}

\begin{prop} \label{prop_C_cap_A}
Let $X$ be a finite $T_0$--space and let $A$ be a bp-retract of $X$. Let $r\colon X\to A$ be the corresponding retraction.
\begin{enumerate}[(a)]
\item If $C\in\C(X)$, then $r(C)=C\cap A$.
\item If $\mxl[X] \subseteq A$ and $C\in\U(X)$, then $C\cap A\in\U(A)$.
\item If $\mnl[X] \subseteq A$ and $C\in\F(X)$, then $C\cap A\in\F(A)$.
\end{enumerate}
\end{prop}

\begin{proof}\ 

($a$) Suppose that $A$ is a dbp-retract of $X$. Clearly, $C\cap A \subseteq r(C)$ and $r(C)\subseteq A$. It remains to prove that $r(C)\subseteq C$. 

Suppose first that $C\in\U(X)$. Let $z\in C$. Since $r(z)\leq z$ and $C$ is an open subset of $X$ it follows that $r(z)\in C$.

Suppose now that $C\in\F(X)$. Let $z\in C$. Note that $z\geq y$ for all $y\in\lb C$. Then $r(z)\geq r(y)=y$ for all $y\in\lb C$, that is, $r(z)\in F_{\lb C}$. Since $r(z)\leq z$ it follows that $C\cup\{r(z)\}$ is a connected subspace of $F_{\lb C}$. Thus $C\cup\{r(z)\}\subseteq C$ and hence $r(z)\in C$.

($b$) Note that $\mxl[A]=\mxl[X]$ and that $C\cap A$ is connected and non-empty by item ($a$). Since $C\in\U(X)$, $C$ is a connected component of $U^X_{\ub C}$. Then $C\cap A\subseteq U^X_{\ub C}\cap A=U^{A}_{\ub C}$.

Thus, there exists a connected component $C_0$ of $U^{A}_{\ub C}$ such that $C\cap A\subseteq C_0$. Then $C_0 \subseteq U^{X}_{\ub C}$ and since $\varnothing\neq C\cap A\subseteq C\cap C_0$ and $C$ is a connected component of $U^X_{\ub C}$ we obtain that $C_0\subseteq C$.
Hence, $C\cap A\subseteq C_0 \subseteq C\cap A$. Therefore $C\cap A= C_0 \in \U(A)$.

($c$) Analogous to that of item ($b$).
\end{proof}

\begin{prop} \label{prop_U_beat_points}
Let $X$ be a finite $T_0$--space and let $A$ be a bp-retract of $X$. Let $r\colon X\to A$ be the corresponding retraction and let $i\colon A \to X$ be the inclusion map.
\begin{enumerate}[(a)]
\item If $\mxl[X]\subseteq A$ then $\U(r)(C)=r(C)$ for all $C\in\U(X)$ and $\U(r)\U(i)=\id_{\U(A)}$.
\item If $A$ is an ubp-retract of $X$ then $\U(r)$ is a homeomorphism.
\item If $\mxl[X] \subseteq A$ and $A$ is a dbp-retract of $X$ then $\U(i)\U(r)\leq \id_{\U(X)}$.
\end{enumerate}
\end{prop}

\begin{proof}\ 

($a$) Let $C\in\U(X)$. Then $r(C)=C\cap A\in \U(A)$ by \ref{prop_C_cap_A}. It follows that $\U(r)(C)=r(C)$.

Let $D\in \U(A)$. Then
\begin{align*}
D &= ri(D)\subseteq r(\U(i)(D)) \subseteq \U(r)(\U(i)(D)) = \U(i)(D)\cap A \subseteq U^X_{\ub D} \cap A = U^{A}_{\ub D}
\end{align*}
by \ref{rem_def_de_Uf_and_Ff}. Since $D$ is a connected component of $U^{A}_{\ub D}$ by \ref{coro_cc_of_U_Bsharp_F_Bflat} and $\U(r)\U(i)(D)$ is connected, it follows that $\U(r)\U(i)(D)=D$.

($b$) Since $A$ is an ubp-retract then $\mxl[X]\subseteq A$ and hence $\U(r)\U(i)=\id_{\U(A)}$ by item ($a$). Let $C\in\U(X)$. Then
\begin{align*}
\U(i)\U(r)(C)=\U(i)(r(C))=\U(ir)(C).
\end{align*}
Note that $ir(C)=C\cap A$ and since $C\in\U(X)$ it follows that $\U(ir)(C) \subseteq C$.

Let $z\in C$. Since $r(z)\geq z$, $r(z)\in ir(C)\subseteq \U(ir)(C)$ and $\U(ir)(C)$ is an open subset of $X$, we obtain that $z\in\U(ir)(C)$. Hence, $\U(ir)(C) = C$.

Therefore, $\U(i)\U(r)(C)=C$.

($c$) Let $C\in\U(X)$. Then
\begin{align*}
\U(i)\U(r)(C)=\U(i)(r(C))=\U(ir)(C) \leq \U(\id_X)(C)=C
\end{align*}
by \ref{prop_f_leq_g_Uf_leq_Ug}.
\end{proof}

Clearly, an analogous result holds for $\F$.

\begin{prop} \label{prop_F_beat_points}
Let $X$ be a finite $T_0$--space and let $A$ be a bp-retract of $X$. Let $r\colon X\to A$ be the corresponding retraction and let $i\colon A \to X$ be the inclusion map.
\begin{enumerate}[(a)]
\item If $\mnl[X]\subseteq A$ then $\F(r)(C)=r(C)$ for all $C\in\F(X)$ and $\F(r)\F(i)=\id_{\F(A)}$.
\item If $A$ is a dbp-retract of $X$ then $\F(r)$ is a homeomorphism.
\item If $\mnl[X]\subseteq A$ and $A$ is an ubp-retract of $X$ then $\F(i)\F(r)\geq \id_{\F(X)}$.
\end{enumerate}
\end{prop}

\begin{proof}
Similar to the proof of \ref{prop_U_beat_points}.
\end{proof}

\begin{rem}
Let $X$ be a finite $T_0$--space and let $A$ be a dbp-retract of $X$ such that $\mxl[X]\subseteq A$. Let $r$ and $i$ be as in \ref{prop_U_beat_points}. Then $\U(r)\U(i)=\id_{\U(A)}$ and $\U(i)\U(r)\leq \id_{\U(X)}$. Thus, $\U(i)$ is an embedding, and applying \ref{theo_dbpr_equivalences} it follows that $\U(i)(\U(A))$ is a dbp-retract of $\U(X)$, that is, $\U(i)(\U(A))$ can be obtained from $\U(X)$ by removing down beat points. Observe that $\U(i)(\U(A))$ is homeomorphic to $\U(A)$.
\end{rem}

The following examples show that several of the hypotheses of the previous propositions are necessary.

\begin{ex}
Let $X$ be the finite $T_0$--space given in figure \ref{fig_ex_3-X}. Note that 4 is a down beat point of $X$. The spaces $\U(X)$, $X-\{4\}$ and $\U(X-\{4\})$ are given in figures \ref{fig_ex_3-UX}, \ref{fig_ex_3b-X} and \ref{fig_ex_3b-UX}, respectively.

Observe that $\{0,2,4\}\in \U(X)$ but $\{0,2\}\notin \U(X-\{4\})$. Thus, \ref{prop_C_cap_A} might not hold if $\mxl[X]\not\subseteq A$.

Let $i\colon X-\{4\}\to X$ be the inclusion map and let $r\colon X\to X-\{4\}$ be the retraction. Note that $\U(r)(\{0,2,4\})=\{0,1,2,3\}\neq r(\{0,2,4\})$ and that $\U(i)\U(r)(\{0,2,4\})=\{0,1,2,3\}\not\leq\{0,2,4\}$. Hence, items ($a$) and ($c$) of proposition \ref{prop_U_beat_points} might not hold if $\mxl[X]\not\subseteq A$.

\begin{figure}
\begin{subfigure}[c]{0.24\textwidth}
\centering
\begin{tikzpicture}
\tikzstyle{every node}=[font=\scriptsize]
\draw (0.5,0) node(0){$\bullet$} node[below=1]{0}; 
\foreach \x in {1,2} \draw (\x-1,1) node(\x){$\bullet$} node[left=1]{\x}; 
\foreach \x in {3,4} \draw (\x-3,2) node(\x){$\bullet$} node[left=1]{\x}; 

\foreach \x in {1,2} \draw (0)--(\x);
\foreach \x in {3} \draw (1)--(\x);
\foreach \x in {3,4} \draw (2)--(\x);
\end{tikzpicture}
\caption{$X$}
\label{fig_ex_3-X}
\end{subfigure}
\begin{subfigure}[c]{0.24\textwidth}
\centering
\begin{tikzpicture}[x=2.1cm, y=1.7cm]
\path node (U3) at (0,1) [CC] {3\\1 2\\0}
 node (U4) at (1,1) [CC] {4\\2\\0}
 node (U34) at (0.5,0) [CC] {2\\0};
 
\draw [-] (U34) to (U3); 
\draw [-] (U34) to (U4);
\end{tikzpicture}
\caption{$\U(X)$}
\label{fig_ex_3-UX}
\end{subfigure}
\begin{subfigure}[c]{0.24\textwidth}
\centering
\begin{tikzpicture}
\tikzstyle{every node}=[font=\scriptsize]
\draw (0.5,0) node(0){$\bullet$} node[below=1]{0}; 
\foreach \x in {1,2} \draw (\x-1,1) node(\x){$\bullet$} node[left=1]{\x}; 
\foreach \x in {3} \draw (\x-3,2) node(\x){$\bullet$} node[left=1]{\x}; 

\foreach \x in {1,2} \draw (0)--(\x);
\foreach \x in {3} \draw (1)--(\x);
\foreach \x in {3} \draw (2)--(\x);
\end{tikzpicture}
\caption{$X-\{4\}$}
\label{fig_ex_3b-X}
\end{subfigure}
\begin{subfigure}[c]{0.22\textwidth}
\centering
\begin{tikzpicture}[x=2.3cm, y=2cm]
\path node (U3) at (0,1) [CC] {3\\1 2\\0};

\end{tikzpicture}
\caption{$\U(X-\{4\})$}
\label{fig_ex_3b-UX}
\end{subfigure}
\caption{ }
\end{figure}
\end{ex}

\begin{ex}
Let $X$ be the finite $T_0$--space given in figure \ref{fig_ex_4-X}. Note that 1 is a down beat point of $X$. The spaces $\U(X)$, $X-\{1\}$ and $\U(X-\{1\})$ are given in figures \ref{fig_ex_4-UX}, \ref{fig_ex_4b-X} and \ref{fig_ex_4b-UX}, respectively.

Note that $\U(X)$ and $\U(X-\{1\})$ are not homeomorphic. Hence, item ($b$) of \ref{prop_U_beat_points} might not hold if $A$ is a dbp-retract of $X$.
\begin{figure}
\begin{subfigure}[c]{0.24\textwidth}
\centering
\begin{tikzpicture}
\tikzstyle{every node}=[font=\scriptsize]
\draw (1,0) node(0){$\bullet$} node[below=1]{0}; 
\draw (1.5,1) node(1){$\bullet$} node[left=1]{1}; 
\foreach \x in {2,3,4} \draw (\x-2,2) node(\x){$\bullet$} node[left=1]{\x}; 

\foreach \x in {1,2} \draw (0)--(\x);
\foreach \x in {3,4} \draw (1)--(\x);
\end{tikzpicture}
\caption{$X$}
\label{fig_ex_4-X}
\end{subfigure}
\begin{subfigure}[c]{0.28\textwidth}
\centering
\begin{tikzpicture}[x=1.5cm, y=1.3cm]
\path node (U2) at (0,2) [CC] {2\\0}
node (U3) at (1,2) [CC] {3\\1\\0}
node (U4) at (2,2) [CC] {4\\1\\0}
node (U34) at (1.5,1) [CC] {1\\0}
node (U234) at (1,0) [CC] {0};
 
\draw [-] (U34) to (U3); 
\draw [-] (U34) to (U4);
\draw [-] (U234) to (U2);
\draw [-] (U234) to (U34);
\end{tikzpicture}
\caption{$\U(X)$}
\label{fig_ex_4-UX}
\end{subfigure}
\begin{subfigure}[c]{0.22\textwidth}
\centering
\begin{tikzpicture}[x=0.9cm, y=0.7cm]
\tikzstyle{every node}=[font=\scriptsize]
\draw (1,0) node(0){$\bullet$} node[below=1]{0}; 
\foreach \x in {2,3,4} \draw (\x-2,2) node(\x){$\bullet$} node[left=1]{\x}; 

\foreach \x in {2,3,4} \draw (0)--(\x);
\end{tikzpicture}
\caption{$X-\{1\}$}
\label{fig_ex_4b-X}
\end{subfigure}
\begin{subfigure}[c]{0.22\textwidth}
\centering
\begin{tikzpicture}[x=1.1cm, y=0.8cm]
\path node (U2) at (0,2) [CC] {2\\0}
node (U3) at (1,2) [CC] {3\\0}
node (U4) at (2,2) [CC] {4\\0}
node (U234) at (1,0) [CC] {0};

\draw [-] (U234) to (U2);
\draw [-] (U234) to (U3);
\draw [-] (U234) to (U4);
\end{tikzpicture}
\caption{$\U(X-\{1\})$}
\label{fig_ex_4b-UX}
\end{subfigure}
\caption{ }
\end{figure}
\end{ex}

\begin{prop}
Let $X$ be a finite $T_0$--space and let $a\in\mxl[X]$ be a down beat point of $X$. Let $b=\max \widehat U_a$.
\begin{enumerate}[(a)]
\item If $b\notin\mxl[X-\{a\}]$ then $U_a$ is a down beat point of $\U(X)$ and $\U(X-\{a\})$ is a ubp-retract of $\U(X)-\{U_a\}$.
\item If $b\in\mxl[X-\{a\}]$ then $\U(X)$ and $\U(X-\{a\})$ are homeomorphic.
\end{enumerate}
\end{prop}

\begin{proof}\ 

($a$) Since $b\notin\mxl[X-\{a\}]$ then there exists $d\in\mxl[X]$ such that $d>b$ and $d\neq a$. Hence $U_b=U_{ad}$ and thus $U_b\in \U(X)$. Clearly $U_b < U_a$. We will prove that $U_b = \max \widehat U^{\U(X)}_{U_a}$. Let $C\in \U(X)$ be such that $C<U_a$. Clearly $a\notin C$ (since otherwise $U_a\subseteq C$). Thus, $C\subseteq \widehat U_a = U_b$. Therefore, $U_a$ is a down beat point of $\U(X)$.

Let $i\colon X-\{a\}\to X$ be the inclusion map and let $r\colon X\to X-\{a\}$ be the retraction which corresponds to the removal of the beat point $a$. Note that $\U(X-\{a\})\subseteq \U(X)-\{U_a\}$, and hence $\U(i)(D)=D$ for all $D\in\U(X-\{a\})$. Then $\U(i)\colon \U(X-\{a\})\to \U(X)$ is the inclusion map.

Let $D\in\U(X-\{a\})$. Note that $\U(r)\U(i)(D)=\U(r)(D)=D$ since $a\notin D$. Hence $\U(r)\U(i)=\id_{\U(X-\{a\})}$.

Let $C\in \U(X)-\{U_a\}$. Applying \ref{prop_C_cap_A} we obtain that $C=C-\{a\}=r(C)\subseteq\U(r)(C)=\U(i)\U(r)(C)$. From  
\ref{theo_dbpr_equivalences} it follows that $\U(X-\{a\})$ is a ubp-retract of $\U(X)-\{U_a\}$.

($b$) Since $b\in\mxl[X-\{a\}]$, $b$ is an up beat point of $X$. Hence $\U(X)$ is homeomorphic to $\U(X-\{b\})$ by \ref{prop_U_beat_points}, which is homeomorphic to $\U(X-\{a\})$ since $X-\{b\}$ and $X-\{a\}$ are homeomorphic.
\end{proof}

\begin{prop}
Let $X$ be a finite $T_0$--space. Then $\U(X)$ does not have up beat points and $\F(X)$ does not have down beat points.
\end{prop}

\begin{proof}
Suppose that $C$ is an up beat point of $\U(X)$. Let $D=\min \widehat F^{\U(X)}_C$. Then 
\begin{displaymath}
\mxl[\U(X)]\cap F^{\U(X)}_D = \mxl[\U(X)]\cap \widehat F^{\U(X)}_C = \mxl[\U(X)]\cap F^{\U(X)}_C
\end{displaymath}
Thus, by \ref{prop_mnl_mxl_de_C(X)},
\begin{align*}
\ub C &= \{a \in \mxl[X] \mid  U_a\in F^{\U(X)}_C\}  = \{a \in \mxl[X] \mid  U_a\in F^{\U(X)}_D\} = \ub D
\end{align*}
Then, by \ref{coro_cc_of_U_Bsharp_F_Bflat}, $C$ and $D$ are connected components of $U_{\ub C}$ and since $C\subseteq D$, it follows that $C=D$, which entails a contradiction.

The second part of the proposition can be proved in a similar way.
\end{proof}

\section{Properties of $\U(X)$}

\begin{prop}\label{prop_Xprime_ubpr}
Let $X$ be a finite $T_0$--space such that every $C\in \U(X)$ has a maximum element. Let $X'=\{\max C \mid C\in \U(X)\}$. Then $X'$ is an ubp-retract of $X$.
\end{prop}

\begin{proof}
Let $i\colon X'\to X$ be the inclusion map. For each $x\in X$ let $C(x)=\min \{C\in \U(X) \mid x\in C \}$. Let $r\colon X \to X'$ be defined by $r(x)=\max C(x)$. 

We claim that $r$ is continuous. Indeed, let $x_1,x_2\in X$ such that $x_1\leq x_2$. Since $C(x_2)$ is an open subset of $X$ we obtain that $x_1\in C(x_2)$ and thus $C(x_1)\subseteq C(x_2)$. Hence, $r(x_1)\leq r(x_2)$.

Clearly, $ir\geq \id_X$. We will prove now that $ri=\id_{X'}$. Let $z\in X'$. Then there exists $C\in \U(X)$ such that $z=\max C$. Since $C(z)$ is an open subset of $X$ we obtain that $U_z\subseteq C(z) \subseteq C \subseteq U_z$. It follows that $ri(z)=z$.

Therefore, $X'$ is an ubp-retract of $X$ by \ref{theo_dbpr_equivalences}.
\end{proof}

\begin{rem}
Observe that, under the hypotheses of the previous proposition, from \ref{prop_U_beat_points} and \ref{prop_Xprime_ubpr} we obtain that $\U(X')\cong \U(X)$.
\end{rem}

\begin{prop} \label{prop_UX_iso_a_Xprime}
Let $X$ be a finite $T_0$--space such that every $C\in \U(X)$ has a maximum element. Let $X'=\{\max C \mid C\in \U(X)\}$. 
Then there exists mutually inverse isomorphisms $\varphi\colon \U(X) \to X'$ and $\psi \colon X' \to \U(X)$ given by $\varphi(C)=\max C$ and $\psi(z)=U^X_z$.
\end{prop}

\begin{proof}
Observe that if $C\in \U(X)$ and $\max C=z$ then, as in the proof of \ref{prop_Xprime_ubpr}, we obtain that $C=U^X_z$. The result follows.
\end{proof}

\begin{lemma} \label{lemma_CC_of_UX}
Let $X$ be a finite $T_0$--space. Let $B\subseteq \mxl[\U(X)]$ be a non-empty subset. Let $A=\{a\in\mxl[X]\mid U^X_a\in B\}$. Then 
\begin{enumerate}[(a)]
\item $\mxl[U^{\U(X)}_B]=\mathcal{C}(U^X_A)$, and
\item $\mathcal{C}(U^{\U(X)}_B)=\{U_M\mid M\in\mathcal{C}(U^X_A)\}$.
\end{enumerate}
\end{lemma}

\begin{proof}\ 

($a$) First note that $B=\{U^X_a \mid a \in A\}$ by \ref{prop_mnl_mxl_de_C(X)}.

Let $D\in \mxl[U^{\U(X)}_B]$. Then $D\subseteq U^X_A$. Since $D$ is connected, there exists a connected component $D_0\in\mathcal{C}(U^X_A)$ such that $D\subseteq D_0$. It follows that $D_0\in U^{\U(X)}_B$. From the maximality of $D$ we obtain that $D=D_0$ and hence $D\in \mathcal{C}(U^X_A)$.

Now, let $D\in \mathcal{C}(U^X_A)$. Then, $D\in U^{\U(X)}_B$. Let $C\in U^{\U(X)}_B$ such that $D\subseteq C$. Then $C\subseteq U^X_A$, and since $C$ is connected it follows that $C\subseteq D$. Thus, $C=D$. Therefore, $D\in \mxl[U^{\U(X)}_B]$.

($b$) From ($a$) we obtain that 
\begin{align*}
\displaystyle U^{\U(X)}_B=\bigcup_{M\in \mathcal{C}(U^X_A)} U^{\U(X)}_M.
\end{align*}
Now, observe that if $M_1,M_2\in \mathcal{C}(U^X_A)$ are such that $M_1\neq M_2$ then $U^{\U(X)}_{M_1}\cap U^{\U(X)}_{M_2}=\varnothing$ since $M_1\cap M_2 = \varnothing$. The result follows.
\end{proof}

\begin{prop} \label{prop_UUX_isomorphic_to_UX}
Let $X$ be a finite $T_0$--space. Then there exists mutually inverse isomorphisms $\varphi\colon \U(\U(X)) \to \U(X)$ and $\psi \colon \U(X) \to \U(\U(X))$ given by $\varphi(D)=\max D$ and $\psi(C)=U^{\U(X)}_C$.
\end{prop}

\begin{proof}
First observe that every $C\in \U(\U(X))$ has a maximum element by item ($b$) of \ref{lemma_CC_of_UX}. Now, if $C\in\U(X)$ then $C$ is a connected component of $U^X_{\ub C}$ and hence $U_C$ is a connected component of $U^{\U(X)}_{\{U_a\mid a \in \ub C\}}$ by item ($b$) of \ref{lemma_CC_of_UX}. Therefore, $\{\max D \mid D\in \U(\U(X))\}=\U(X)$. The result follows from \ref{prop_UX_iso_a_Xprime}.
\end{proof}

\begin{coro}
Let $X$ be a finite $T_0$--space. Then there exists a finite $T_0$--space $Y$ such that $X\cong \U(Y)$ if and only if $X\cong \U(X)$.
\end{coro}

\begin{proof}
Follows easily from \ref{prop_Uf_homeo} and \ref{prop_UUX_isomorphic_to_UX}.
\end{proof}

\section{Other properties of the $\C$--construction}

In this section we will develop more results about the $\C$--construction. Not all of them are related to the fixed point property but are interesting on their own and permit a better understanding of this construction. In addition, these results show that the $\C$--construction is an interesting object of study.

In the following proposition we will see that if $X$ is a finite $T_0$--space, then $X$ is homotopy equivalent to the Grothendieck construction on a functor from $\U(X)$ to the category of finite posets, or equivalently to the the non-Hausdorff homotopy colimit of a $\U(X)$--diagram of finite posets \cite{fernandez2016homotopy}.

Recall that if $A$ is a finite poset, $\mathcal{P}$ is the category of finite posets and $F\colon A\to \mathcal{P}$ a functor, then the Grothendieck construction for $F$ is the poset $\int F$ with underlying set $\bigcup\limits_{a\in A} (\{a\}\times F(a))$ and with partial order defined as follows: if $a,b\in A$, $x\in F(a)$, $y\in F(b)$ then $(a,x)\leq (b,y)$ if and only if $a\leq b$ and $F(a\leq b)(x)\leq y$.

\begin{definition}
Let $X$ be a finite $T_0$--space. We define the functor $\mathcal{I}_X\colon \U(X)\to \mathcal{P}$ by $\mathcal{I}_X(C)=C$ and $\mathcal{I}_X(C\subseteq C')$ as the inclusion map $C\hookrightarrow C'$. 

We also define maps $\rho_X\colon \hocolim \mathcal{I}_X \to X$ and $\iota_X\colon X\to \hocolim \mathcal{I}_X$ by $\rho_X(C,x) = x$ for all $(C,x)\in \hocolim \mathcal{I}_X$ and $\iota(x)=(\CU(x), x)$ for all $x\in X$.
\end{definition}

\begin{prop} \label{prop_X_dbpr_hocolim}
Let $X$ be a finite $T_0$--space. Then, the maps $\rho_X\colon \hocolim \mathcal{I}_X \to X$ and $\iota_X\colon X\to \hocolim I_X$ are continuous, $\rho_X \iota_X=\id_X$ and $\iota_X\rho_X \leq\id_{\hocolim \mathcal{I}_X}$.

In particular, $\iota_X(X)$ is a dbp--retract of $\hocolim \mathcal{I}_X$ and $X$ is homotopy equivalent to $\hocolim \mathcal{I}_X$.
\end{prop}

\begin{proof}
Let $(C_1,x_1),(C_2,x_2)\in\hocolim \mathcal{I}_X$ such that $(C_1,x_1)\leq (C_2,x_2)$. Then $C_1\subseteq C_2$ and $x_1=\inc(x_1)\leq x_2$ in $C_2$. Hence $\rho_X(C_1,x_1) = x_1\leq x_2 = \rho_X(C_2,x_2)$ in $X$. Thus, $\rho_X$ is continuous.
 
Now, let $x_1,x_2\in X$ such that $x_1\leq x_2$. Since $x_2\in \CU(x_2)$ and $\CU(x_2)$ is an open subset of $X$, then $x_1\in \CU(x_2)$. Hence, $\CU(x_1)\leq \CU(x_2)$ and thus $\iota_X(x_1)=(\CU(x_1),x_1)\leq (\CU(x_2),x_2)=\iota(x_2)$. Therefore, $\iota_X$ is continuous. 

Now, for all $x\in X$, $\rho_X \iota_X(x)=\rho_X(\CU(x),x)=x$. Hence $\rho_X \iota_X=\id_X$.
 
On the other hand, if $(C,x)\in\hocolim \mathcal{I}_X$, then $x\in C$ and $\iota_X \rho_X(C,x)=\iota_X(x)=(\CU(x),x)\leq (C,x)$. Hence $\iota_X \rho_X \leq \id_{\hocolim \mathcal{I}_X}$.
\end{proof}

\begin{rem}
Note that if $X$ is a finite $T_0$--space then $\mxl[\hocolim \mathcal{I}_X]=\{(U_a,a)\mid a\in \mxl[X]\}$.
\end{rem}

\begin{lemma}
Let $X$ be a finite $T_0$--space. For each $D\in\U(X)$ let $\widetilde D=\{(C,x)\in \hocolim \mathcal{I}_X \mid C \subseteq D\}$.
\begin{enumerate}[(a)]
\item Let $B\subseteq\mxl[\hocolim \mathcal{I}_X]$. Let $A=\{a\in\mxl[X]\mid (U_a,a)\in B\}$. Then 
\begin{displaymath}
\mathcal{C}(U_{B}^{\hocolim \mathcal{I}_X})=\{\widetilde D \mid D\in\mathcal{C}(U_A^X)\}.
\end{displaymath}
\item Let $C\in\U(X)$. Then $\U(\iota_X)(C)=\widetilde C$.
\end{enumerate}
\end{lemma}

\begin{proof}\ 

($a$) Note that $B=\{(U_a,a)\mid a\in A\}$ by the previous remark.

We will prove first that $U_{B}^{\hocolim \mathcal{I}_X}=\bigcup\limits_{D\in\mathcal{C}(U_A^X)}\widetilde D$. 

Let $(C,x)\in U_{B}^{\hocolim \mathcal{I}_X}$. Then $C\subseteq U_a$ for all $a\in A$. Hence, $C\subseteq U^X_A$, and since $C$ is connected, there exists $D\in \mathcal{C}(U_A^X)$ such that $C\subseteq D$. Thus, $(C,x)\in \widetilde D$. 

Now, let $D\in\mathcal{C}(U_A^X)$ and let $(C,x)\in \widetilde D$. Then $D\subseteq U_a$ for all $a\in A$. Hence $(C,x) \leq (U_a,a)$ for all $a\in A$. Thus, $(C,x) \in U_{B}^{\hocolim \mathcal{I}_X}$.

We claim that if $D_1,D_2$ are distinct elements of $\mathcal{C}(U_A^X)$ then $\widetilde D_1\cap \widetilde D_2 = \varnothing$. Indeed, if there exists $(C,x)\in \widetilde D_1\cap \widetilde D_2$, then $C\subseteq D_1\cap D_2$ which entails a contradiction since $D_1$ and $D_2$ are different connected components of $U^X_A$.

We will prove now that $\widetilde D$ is connected for all $D\in\mathcal{C}(U_A^X)$. Let $D\in\mathcal{C}(U_A^X)$. Note that
\begin{align*}
\widetilde D &= \bigcup_{x\in D} U^{\hocolim \mathcal{I}_X}_{(D,x)}= \bigcup_{x\in D} \left( U^{\hocolim \mathcal{I}_X}_{(D,x)} \cup \{D\}\times D \right)
\end{align*}
Now, for each $x\in D$, $U^{\hocolim \mathcal{I}_X}_{(D,x)} \cup \{D\}\times D$ is a connected subspace of $\hocolim \mathcal{I}_X$ since $(D,x)\in U^{\hocolim \mathcal{I}_X}_{(D,x)} \cap \{D\}\times D$ and $U^{\hocolim \mathcal{I}_X}_{(D,x)}$ and $\{D\}\times D$ are connected. Therefore, $\widetilde D$ is also connected.

Finally, note that $\widetilde D$ is an open subset of $\hocolim \mathcal{I}_X$ for all $D\in \mathcal{C}(U_A^X)$. Since $\hocolim \mathcal{I}_X$ is a finite space, the result follows.

($b$) Let $C\in\U(X)$. If $x\in C$ then $\iota_X(x)=(\CU(x),x)\in \widetilde C$ since $\CU(x)\subseteq C$. Thus, $\iota_X(C)\subseteq\widetilde C$.

We will prove now that $\ub{\iota_X(C)} = \ub{\widetilde C}$. Clearly, $\ub{\iota_X(C)} \supseteq \ub{\widetilde C}$ since $\iota_X(C)\subseteq\widetilde C$. Let $M\in \ub{\iota_X(C)}$. Then there exists $a\in \mxl[X]$ such that $M=(U_a,a)$. Let $x\in C$. Then $\iota_X(x)=(\CU(x),x)\leq (U_a,a)$ and hence $x\in \CU(x)\subseteq U_a$. Therefore, $C\subseteq U_a$. Now, let $(C_1,x_1)\in\widetilde C$. Then $x_1\in C_1\subseteq C\subseteq U_a$. Hence, $(C_1,x_1)\leq (U_a,a)=M$. Thus, $M\in \ub{\widetilde C}$.

Now, by \ref{coro_cc_of_U_Bsharp_F_Bflat}, $\widetilde C$ is a connected component of $U^{\hocolim \mathcal{I}_X}_{\ub{\widetilde C}}$. Since $\ub{\iota_X(C)} = \ub{\widetilde C}$ and $\iota_X(C)\subseteq\widetilde C$ we obtain that $\U(\iota_X)(C)=\widetilde C$ by \ref{prop_conexo_minimal}.
\end{proof}

\begin{prop}
Let $X$ be a finite $T_0$--space. Then $\U(\hocolim \mathcal{I}_X)\cong \U(X)$. Moreover, $\U(\iota_X)\colon \U(X)\to \U(\hocolim \mathcal{I}_X)$ is a homeomorphism with inverse $\U(\rho_X)$.
\end{prop}

\begin{proof}
Let $C\in\U(X)$. Then $\U(\iota_X)(C)=\{(D,x)\in \hocolim \mathcal{I}_X \mid D\subseteq C\}$ by the previous lemma. Thus,
\begin{align*}
\rho_X(\U(\iota_X)(C))=\bigcup_{\substack{D\in\U(X)\\ D\subseteq C}}D = C.
\end{align*}
Hence, $\U(\rho_X)(\U(\iota_X)(C))=C$. Therefore $\U(\rho_X)\U(\iota_X)=\id_{\U(X)}$.

Observe that $\U(\iota_X)$ is surjective by the previous lemma. The result follows.
\end{proof}

\begin{definition}
Let $X$ be a finite $T_0$--space. We define $q_X\colon \hocolim \mathcal{I}_X \to \U(X)$ by $q_X(C,x)=C$.
\end{definition}

\begin{prop}
Let $X$ be a finite $T_0$--space.
\begin{enumerate}
\item $q_X$ is a quotient map.
\item Let $\varphi$ and $\psi$ be the isomorphisms of proposition \ref{prop_UUX_isomorphic_to_UX}. Then the following diagram commutes
\[
\begin{tikzcd}[column sep=large, row sep=large]
\U(\hocolim \mathcal{I}_X) \rar{\U(q_X)} \drar[swap]{\U(\rho_X)} &\U(\U(X)) \dar[xshift=2pt]{\varphi}\\
&\U(X) \uar[xshift=-2pt]{\psi}
\end{tikzcd}
\]
\end{enumerate}
\end{prop}

From \ref{prop_Xprime_ubpr} and \ref{prop_UX_iso_a_Xprime} we obtain that $\U(X)$ and $X$ are homotopy equivalent if every $C\in \U(X)$ has a maximum element. The following proposition provides a version of this result for weak equivalences. 

\begin{prop}
Let $X$ be a finite $T_0$--space. Suppose that every $C\in \U(X)$ is weakly contractible. Then there exists a weak homotopy equivalence $X\to \U(X)$.
\end{prop}

\begin{proof}
Let $\mathcal{J}\colon \U(X)\to \mathcal{P}$ be the constant functor with value $\ast$, the one-element poset. Let $\alpha\colon \mathcal{I}_X \Rightarrow \mathcal{J}$ be the unique natural transformation. It follows that $\alpha_C$ is a weak homotopy equivalence for every $C\in \U(X)$ and hence, from \cite[Corollary 2.6]{fernandez2016homotopy}, we obtain that there exists a weak homotopy equivalence $\eta\colon \hocolim \mathcal{I}_X \to \hocolim \mathcal{J}$. Clearly $\hocolim \mathcal{J}\cong \U(X)$. Besides, by \ref{prop_X_dbpr_hocolim}, $X$ is homotopy equivalent to $\hocolim \mathcal{I}_X$. The result follows.
\end{proof}

\bibliographystyle{acm}
\bibliography{ref_C_constr}

\end{document}